\documentclass[reqno]{amsart}
\usepackage[all]{xy}
\usepackage{enumerate}
\usepackage{amscd}
\usepackage{enumerate}
\usepackage{amssymb,amsmath}
\usepackage{graphicx}
\theoremstyle{plain}
  \newtheorem{theorem}{Theorem}[section]
  \newtheorem{corollary}[theorem]{Corollary}
  \newtheorem{lemma}[theorem]{Lemma}
  \newtheorem{proposition}[theorem]{Proposition}
  
 \newtheorem{obs}[theorem]{Observation}
  \theoremstyle{definition}
  \newtheorem{definition}[theorem]{Definition}

\DeclareMathOperator{\id}{id}
\DeclareMathOperator{\dom}{dom}
\DeclareMathOperator{\img}{img}
\DeclareMathOperator{\mul}{mul}

\DeclareMathOperator{\GL}{GL}

\DeclareMathOperator{\Id}{Id}

\DeclareMathOperator{\grad}{grad}

\DeclareMathOperator{\coker}{coker}

\newcommand\reg{\mathrm{reg}}

\newcommand\tX{{\tilde X}}
\newcommand\lf{\mathrm{l.f.}}
\newcommand\Nov{\mathrm{Nov}}

\newcommand{\N}{\mathbb N}
\newcommand{\X}{\mathbb X}
\newcommand{\R}{\mathbb R}

\newcommand{\im}{\text{\rm im}}

\begin{document}

\title{Graph Representations  and topology of  real  and angle valued maps}

\author{Dan Burghelea}

\address{Dept. of Mathematics,
         The Ohio State University,
         231 West 18th Avenue,
         Columbus, OH 43210, USA.}

\email{burghele@mps.ohio-state.edu}

\author{Stefan Haller}

\address{Department of Mathematics,
         University of Vienna,
         Nordbergstra{\ss}e 15,
         A-1090 Vienna, Austria.} 

\email{stefan.haller@univie.ac.at}

\thanks{Part of this work
  was done while the second author enjoyed the warm hospitality of the
  Ohio State University. 
  The first author acknowledge partial support from NSF grant MCS 0915996.
  The second author acknowledges the support of the Austrian Science Fund, grant P19392-N13.}



\date{\today}

\begin{abstract}
In this paper we review the definition of the invariants ``bar codes'' and ``Jordan cells'' of real and angle valued tame maps 
as proposed in \cite{BD11} and \cite{CSD09} and prove the homotopy invariance of the sums $ \sharp \mathcal B^c_r +\sharp \mathcal B^o_{r-1}$ 
and of the set of Jordan cells.  Here $\mathcal B^c_r$ resp. $\mathcal B^o_r$ denote the sets of closed resp. open bar codes in dimension $r.$  In addition we provide  calculation of some familiar topological invariants in terms of bar codes and Jordan cells. 
The presentation provides a different perspective on Morse--Novikov theory based on critical values, bar codes and Jordan cells rather than on critical points instantons and closed trajectories of a gradient of a real or angle valued map.
\end{abstract}

\maketitle

\setcounter{tocdepth}{1}
\tableofcontents

\section{Introduction}\label{S:intro}

Recently, using graph representations, a new type of invariants, \emph{bar codes} resp. \emph{bar codes} and \emph{monodromy (Jordan cells)}, have been assigned to a \emph{tame real valued map}
$f\colon X\to\mathbb R$ resp. a \emph{tame angle valued map} $f\colon X\to S^1$ 
and a field $\kappa$ \footnote{ More recent work which will be detailed in \cite {B12} provides a definition of bar codes without any reference to graph representations and extend of  the results below to continuous maps $f$ whose $X$ and the levels of $f$ are all compact ANR's.}.
 They were first introduced in \cite{CSD09} and \cite{BD11} as invariants for zigzag persistence resp.\ persistence for circle valued maps based on the changes in the homology of the fibers with coefficients in $\kappa.$

In this paper we define these invariants, establish additional results which relate them to  familiar topological invariants  and prove the homotopy invariance of the set of Jordan cells and 
of the numbers $ \sharp \mathcal B^c_r +\sharp \mathcal B^o_{r-1}.$ Here $\mathcal B^c_r$ resp.$\mathcal B^o_r$ denote the sets of closed resp. open bar codes in dimension $r.$ 

The main results are contained in Theorems~\ref{T2}, \ref{T3}, \ref{P1} and Corollary  \ref{T4}, presented in section~\ref{S1}.
The theory presented below represents an alternative approach to Morse--Novikov theory for real valued and angle valued maps based  on critical values instead of critical points.  In our approach the topological information about  the underlying space  is derived from  bar codes between critical values,  Jordan cells   and the canonical long exact sequence associated with  a tame map. Morse--Novikov theory cf. \cite {Fa04}, \cite{Pa06}, derives this information  from 
instantons (isolated trajectories) between critical points,  closed trajectories  and the Morse complex associated with the gradient of a Morse (real or circle valued) map on the underlying  Riemannian manifold.  
Our approach  applies to a considerably larger class of continuous maps than the maps considered by Morse--Novikov theory.

The tame real valued maps are  tame angle valued maps and all results about them are particular cases of results about angle valued maps.  Rather than consider only angle valued maps, considerably more complex, we decided  to discuss both cases, simply because the Morse theory of real valued maps is more familiar than Novikov theory of the circle valued maps and restricting the attention only to the second apparently does not save much space. 


The \emph{bar codes} are finite intervals $I$ of real numbers of the type:
\begin{enumerate}
\item closed, $[a, b]$, $a\leq b$,
\item open $(a,b)$, $a<b$, and
\item mixed $[a,b)$, $(a,b]$, $a<b$. 
\end{enumerate}


The \emph{Jordan blocks} $J$ and \emph{Jordan cells} are equivalency classes of pairs $(V,T)$ with $V$ a finite dimensional $\kappa$-vector space and $T\colon V\to V$ 
a linear isomorphism. 

An equivalence between $(V_1,T_1)$ and $(V_2, T_2)$ is an linear isomorphism $\omega\colon V_1\to V_2$ which intertwines $T_1$ and $T_2$.

An equivalence class is called a \emph{Jordan block} if indecomposable, i.e.\ $(V,T)$ is not isomorphic to $(V_1,T_2)\oplus(V_2,T_2)$ with $\dim V_i <\dim V$.

A Jordan block $J= (V,T)$ is called \emph{Jordan cell} if isomorphic to $(\kappa^k,T(\lambda,k))$, $k\in\mathbb N$, $\lambda\in\kappa\setminus 0$, where 
\begin{equation}\label{E1}
T(\lambda,k)=
\begin{pmatrix}
\lambda & 1       & 0      & \cdots  & 0      \\
0       & \lambda & 1      & \ddots  & \vdots \\
0       & 0       & \ddots & \ddots  & 0      \\
\vdots  & \ddots  & \ddots & \lambda & 1      \\
0       & \cdots  & 0      & 0       & \lambda
\end{pmatrix}
\end{equation}
in which case will be denoted by $J=(\lambda,k).$

If $\kappa$ is algebraically closed  then a Jordan block is a Jordan cells and the two concepts are the same \footnote{If $(V,T)$ is a Jordan block then 
$(V\otimes \overline \kappa, T\otimes \overline \kappa)$ is not necessary a Jordan cell but decomposes as a sum of Jordan cells.}.


For a tame real valued map $f\colon X\to\mathbb R$ and $r\leq\dim X$ we associate (see section~\ref{S1}) the collection of bar codes  
$\mathcal B_r(f)$. The set $\mathcal B_r(f)$ can be written as $\mathcal B_r(f)=\mathcal B^c_r(f)\sqcup\mathcal B^o_r(f)\sqcup \mathcal B^m_r(f)$ 
with $\mathcal B^c_r(f)$, $\mathcal B_r^o$ and  $\mathcal B^m_r$ the subset of closed,  open  and  mixed bar codes.


For a tame angle valued map $f$, in addition to bar codes as above, one associates the collections of Jordan blocks $\mathcal J_r(f)$, equivalently  
of Jordan cells $\overline{\mathcal J}_r(f)$ if one consider $\overline\kappa$ the algebraic closure of $\kappa$. The sum $(V_r(f), T_r(f))$ of all 
Jordan blocks in $\mathcal J_r(f)$ or of all Jordan cells in $\overline {\mathcal J}_r(f)$ is referred to as the \emph{monodromy} of $f$.

If $f$ is only a continuous map, in view of Theorem~\ref{T4},  $\mathcal J_r(f)$ resp.\ $\overline{\mathcal J}_r(f)$ can still be defined.  
It is expected (and will be shown in \cite{B12}) that the sets $\mathcal B^c_r(f)$, $\mathcal B^o_r(f).$ 

Theorem~\ref{T4} states  that  the numbers $N_r(f):= \sharp \mathcal B^c_r(f)+\sharp \mathcal B^o_{r-1}(f)$ are homotopy invariants  
of the pair $(X,\xi_f)$ where $\xi_f \in H^1(X;\mathbb Z)$ is the cohomology class determined by $f$.  
We say that the pairs $(X_i,\xi_i\in H^1(X, \mathbb Z))$, $i=1,2$, are homotopy equivalent, if there exists a homotopy equivalence 
$\theta\colon X_1\to X_2$ so that $\theta^\ast(\xi_2)=\xi_1$.  

Theorem~\ref{T4}, also states the homotopy invariance  for the monodromy. In view of these facts we might want to get a  
homotopy-theoretic description of the numbers $N_r(f)$ and of the monodromy $(V_r(f), T_r(f)).$

For this purpose consider $(X,\xi\in H^1(X;\mathbb Z))$ and denote by $\tilde X\to X$ the infinite cyclic cover associated to $\xi$. 
Note that $H_r(\tilde X):= H_r(\tilde X;\kappa)$ is not only a $\kappa$-vector space but is actually a $\kappa[T^{-1}, T]$-module where the 
multiplication by $T$ is induced by the deck transformation $\tau\colon\tilde X\to \tilde X$.
Here $\kappa[T^{-1},T]$ denotes the ring of Laurent polynomials. Let $\kappa[T^{-1},T]]$ be the field of Laurent power series. Define  
\[
H^N_r(X;\xi):=H_r(\tilde X)\otimes_{\kappa[T^{-1}, T]} \kappa[T^{-1}, T]]
\] 
and let 
\[ 
H_r(\tilde X)\to H^N_r(X;\xi)
\] 
be the $\kappa[T^{-1},T]$-linear map induced by taking the tensor product of $H_r(\tilde X)$ with $\kappa[T^{-1},T]]$ over $\kappa[T^{-1},T]$.

The $\kappa[T^{-1}, T]]$-vector spaces $H^N_r(X;\xi)$ are called Novikov homology\footnote{instead of $\kappa[T^{-1},T]]$ one can consider the field 
$\kappa[[T^{-1},T]$ of Laurent power series in $T^{-1}$, which is isomorphic to $\kappa[T^{-1},T]]$ by an isomorphism induced by $T\to T^{-1}$.  
The (Novikov) homology defined using this field has the same Novikov--Betti numbers as the the one defined using $\kappa[T^{-1}, T]]$.}
and their dimensions,  the numbers $N_r(X;\xi):=\dim H^N_r(X;\xi),$ Novikov--Betti numbers.

It $X$ is a compact ANR then 
the $\kappa$-vector space   
$V(\xi):=\ker (H_r(\tilde X)\to H^N_r(X;\xi))$ is finite dimensional and when equipped with  $T(\xi)\colon V(\xi)\to V(\xi)$ induced by the 
multiplication by $T$ defines the pair $(V(\xi),T(\xi))$ called  the monodromy of $(X,\xi)$.  We  show that  the numbers $N_r(f)$ are exactly $N_r(X;\xi_f)$ (cf.\ Theorem~\ref{T2}), and the pair $(V_r(f), T_r(f))$ described 
using graph representations 
is exactly the monodromy $(V(\xi_f),T(\xi_f)).$  

The monodromy  can be defined for an arbitrary continuous map $f\colon X\to S^1$, using instead of graph representations the  
regular part of a linear relation provided by the map $f$ in the homology of any fiber of $f$, as described in section~\ref{S6}.


The plan of this paper is the following.

In section~\ref{S1} we remind the reader the concepts of tame real and angle valued maps and formulate  the main results. 

In sections~\ref{S2} we discuss the representation theory for the two graphs, $\mathcal Z$ and $G_{2m}$, 
used  in the proof of the main theorems. The reader can skip section 4 unless he wants to understand the calculations of the bar codes and the Jordan cells for the example presented in section ~\ref{S7} via an implementable algorithm.   

In sections~\ref{S5} and \ref{S6} we prove the main results. 
Section \ref{S6} can be read independently of the rest of the paper. It does provide the necessary background on linear relations and does not use concepts previously defined.

In section~\ref{S7} we give an example of a tame angle valued map and derive its bar codes and Jordan cells using the algebraic observations made in section~\ref{S2}. 

{\it Acknowledgements:} The relationship between the topology of a space to the information extracted from the real or angle valued map as presented in this paper  was influenced by the {\it persistence theory}  introduced in \cite{ELZ02} and motivated by the interest that computer scientists and data analysts have shown for {\it persistent homology} and associated concepts. It also owns to the apparently forgotten efforts and ideas of R. Deheuvels  to extend Morse theory to all  continuous real valued functions (fonctionelles) cf. \cite{RD55}.

\section{The main results}\label{S1}

\subsection{Tame maps and its $r$-invariants.}\label{SS11}

\begin{definition}
A continuous map $f\colon X\to\mathbb R$ resp.\ $f\colon X\to S^1$, $X$ a compact ANR, is \emph{tame} if:
\begin{enumerate}[1.]
\item
Any fiber $X_\theta=f^{-1}(\theta)$ is the deformation retract of an open neighborhood. 
\item
Away from a finite set of numbers/angles $\Sigma=\{\theta_1,\dots,\theta_m\} \subset \mathbb R$, resp.\ $S^1$
the restriction of $f$ to $X\setminus f^{-1}(\Sigma)$ is a fibration (Hurewicz fibration).
\end{enumerate}
\end{definition}

Note that:  

-  Any smooth real or angle valued map on a compact smooth manifold $M$ whose all critical points are isolated,  in particular  any Morse function, is tame.

-  Any real or angle valued simplicial map on a finite simplicial complex is tame.

-  The space of tame maps with the induced topology has the same homotopy type as the space of all continuous maps with compact open 
topology.\footnote{While (i) and (ii) are simple exercises, we can not locate a reference for statement (iii), but since all ANR's of interest for this 
paper are homeomorphic to simplicial complexes, for them the statement follows from (ii).}

-  The set of tame maps is dense the  the space of all continuous maps with respect to the compact open topology.\footnote{The same comment as in footnote 3.}


Given a tame map $f\colon X\to\mathbb R$ resp.\  $f\colon X\to S^1$ consider the critical values resp.\ the critical angles $\theta_1<\theta_2<\cdots<\theta_m$.  
In the second case  we have $0<\theta_1<\cdots<\theta_m\leq 2\pi$. 
Choose $t_i$, $i=1,2,\dotsc,m$, with $\theta_1<t_1 <\theta_2<\cdots<t_{m-1}<\theta_m<t_m$.  In the second case choose $t_m$ s.t.\ $2\pi<t_m <\theta_1 +2\pi$. 
The tameness of $f$ induces the diagram of continuous maps:\footnote{The doted arrow $a_1$ in Diagram~1 appears only in the case of an angle valued map.}
\begin{center}
{\scriptsize
$$
\xymatrix{
    & \boldsymbol{X}_{t_m}\ar@{.>}[dl]_{a_1} \ar[dr]^{b_m}&     & \boldsymbol{X}_{t_{m-1}}\ar[dl]_{a_m}&\cdots \boldsymbol{X}_{\omega_{2}}\ar[dr]^{b_2}& & \boldsymbol{X}_{t_1}\ar[dl]_{a_{2}} \ar[dr]^{b_1} & \\
 \boldsymbol{X}_{\theta_1} &                &  \boldsymbol{X}_{\theta_m}&&& \boldsymbol{X}_{\theta_{2}}&   & \boldsymbol{X}_{\theta_1}
} _.
$$
}
Diagram~1
\end{center}

Different choices of $t_i$ lead to different diagrams but all homotopy equivalent.

We will use two graphs, $\mathcal Z$ for real valued maps, and $G_{2m}$ for angle valued maps. 
The graph $\mathcal Z$ has vertices $x_i$, $i\in\mathbb Z$, and edges $a_i$ from $x_{2i-1}$ to $x_{2i}$ and $b_i$ from $x_{2i+1}$ to $x_{2i}$,
see picture (The graph $\mathcal Z$) in section~\ref{S2}.


The graph $G_{2m}$ has vertices $x_i$ with edges $a_i$ and $b_i$, $i=1,2,\dotsc,(m-1)$, as before and $b_m$ from $x_1$ to $x_{2m}$, see picture (The graph $G_{2m}$) 
in section~\ref{S2}.


Let $\kappa $ be a field. A graph representation $\rho$ is an assignment which to each vertex $x$ assigns a finite dimensional vector space $V_x$  
and to each oriented arrow from the vertex $x$ to the vertex $y$ a linear map $V_x\to V_y$.

As stated in section~\ref{S2} a finitely supported $\mathcal Z$-representation\footnote{ i.e. all but finitely many vector spaces $V_x$ have dimension zero}, resp.\ an arbitrary $G_{2m}$-representation can be uniquely 
decomposed as a sum of indecomposable representations. In the case of the graph $\mathcal Z$ the indecomposable representations  are indexed by  
one of the three types of intervals (bar codes)  described in the introduction, with ends $i,j\in \mathbb Z$, $i\leq j$ for type (i) and $i<j$ for type (ii) and (iii). 
We refer to both the inedecomposable representation  and the interval  as \emph{bar code.}
In the case of the graph $G_{2m}$ the indecomposable representations  are indexed by similar intervals (bar codes) with ends $i,j +mk, $ $1\leq i,j\leq m, k\in \mathbb Z_{\leq 0}$, $i\leq j$ with $1\leq i\leq m$ 
and by Jordan blocks $J$ (or Jordan cells) as described in the introduction. 
We refer to both the indecomposable representation 
and the interval  resp. 
the Jordan block  as \emph{ bar code} resp. 
\emph{Jordan block} or {\emph Jordan cells}.

For a $\mathcal Z$-representation or a $G_{2m}$-representation $\rho$ one denotes by $\mathcal B(\rho)$ the set of all 
bar codes and write $\mathcal B(\rho)$ as $\mathcal B(\rho)=\mathcal B^c(\rho)\sqcup \mathcal B^o(\rho)\sqcup \mathcal B^m(\rho)$ where 
$\mathcal B^c(\rho)$, $\mathcal B^o(\rho)$ and  $\mathcal B^m(\rho)$ are the subsets of closed, open and mixed bar codes.

For a $G_{2m}$ representation $\rho$ one denotes by $\mathcal J(\rho)$ resp. $\overline{\mathcal J}(\rho)$ the set of all 
Jordan blocks resp. Jordan cells.


For any $r\leq \dim X$ let $\rho_r= \rho(f)$ be the $\mathcal Z$- resp.\ $G_{2m}$-representation associated to the tame map $f$  defined by
$$
V_{2i}= H_r(X_{\theta_i}),V_{2i+1}= H_r(X_{t_{i}}),\quad \alpha_i:V_{2i-1}\to V_{2i},\quad \beta_i:V_{2i+1}\to V_{2i}
$$
with $\alpha_i$ and $\beta_i$ the linear maps induced by the continuous maps $a_i$ and $b_i$ in Diagram~1.  
Here and below $H_r(Y)$ denotes the singular homology in dimension $r$ with coefficients in a fixed field $\kappa $ which will not appear in the notation.

In order to relate the indecomposable components of $\rho_r$ to the critical vales or angles of $f,$ for a real valued map one converts the intervals $\{i,j\}$  into $\{\theta_i,\theta_j\}$   and for an angle value map the intervals 
$\{i,j+km\}$, $1\leq i,j\leq m$,  into the intervals $\{\theta_i,\theta_j+2\pi k\}$ \footnote {we use the symbol "$\{$" for both "$($" and "$[ $" or  "$\}$" for both $")"$ or $" ] ".$}.

\vskip.1in
{\bf Definition:}  The sets $\mathcal B_r(f):= \mathcal B(\rho_r),$ with the intervals $I$ converted into ones with ends $\theta_i'$s and $(\theta_i +2\pi k)'$s and $\mathcal J_r(f):=\mathcal J(\rho_r)$ resp.\ 
$\overline{\mathcal J}_r=\overline{\mathcal J}(\rho_r)$ are the $r$-invariants of the map $f$. 
\vskip .1in
For a real valued map one has only
bar codes,  for an angle valued map one has bar codes  and Jordan blocks or Jordan cells.

We refer to 
$$
(V_r(f),T_r(f))=\bigoplus_{(V,J)\in \mathcal J_r(f)}(V,T)=\bigoplus_{(\lambda,k)\in \overline{\mathcal J}_r(f)}(\kappa^k,T(\lambda,k))
$$ 
as the \emph{$r$-monodromy} of the angle valued $f$.


Recall that the homotopy classes of continuous maps $f\colon X\to S^1$ are in bijective correspondence to $H^1(X;\mathbb Z)$ so any such map 
$f$ defines $\xi:=\xi_f \in H^1(X;\mathbb Z)$  and any homotopy class  can be viewed as an element in $H^1(X;\mathbb Z)$. 

\begin{definition}
1. Two maps $f_1\colon X_1\to S^1$ and $f_2\colon X_2\to S^1$ or $f_1\colon X_1\to\mathbb R$ and $f_2\colon X_2\to \mathbb R$ 
are fiber wise homotopy equivalent if there exists $\omega\colon X_1\to X_2$ so that 
$f_2\cdot\omega=f_1$ and for any $\theta\in S^1$ the restriction $\omega_\theta\colon(X_1)_\theta\to(X_2)_\theta$ is a homotopy equivalence.

2. Two maps $f_1\colon X_1\to S^1$ and $f_2\colon X_2\to S^1$ are homotopy equivalent if there exists $\omega\colon X_1\to X_2$  so that 
$f_2\cdot\omega$ is homotopic to $f_1$, equivalently $\omega^{\ast}(\xi_2)=\xi_{1}$, $\xi_i=\xi_{f_i}$. If so we say that the pairs 
$(X_1,\xi_1)$ and $(X_2,\xi_2)$ are homotopy equivalent.
\end{definition}

The following statement follows from   definitions.

\begin{proposition}
If $f_i\colon X_i\to \mathbb R$ resp.\ $f_i\colon X_i\to S^1$, $i=1,2$, are two tame maps and $\omega\colon X_1\to X_2$ is
a fiber wise homotopy equivalence then 
$\mathcal B_r(f_1)=\mathcal B_r(f_2)$ resp.\ $\mathcal B_r(f_1)=\mathcal B_r(f_2)$ and 
$\mathcal J_r(f_1)=\mathcal J_r(f_2)$ (equivalently $\overline{\mathcal J}_r(f_1)=\overline{\mathcal J_r}(f_2)$).
\end{proposition}

\vskip .1in

\subsection{The results.}\label{SS12}

Fix a field $\kappa$ and denote by $H_\ast (Y)$ the singular homology of $Y$ with coefficients  in the field $\kappa.$  The following result was established in \cite{BD11}.


\begin{theorem}[\cite{BD11}]\label{T1}

1. If $f\colon X\to S^1$ is a tame map then:
\begin{align*}
\hskip .4in a. &\dim H_r(X_\theta)
    &&=\sum_{I\in \mathcal B_r(f)}n_\theta(I)+\sum_{J\in\mathcal J_r(f)}k(J)
\\
\hskip .4in b.&\dim H_r(X)
&&=\begin{cases}
\sharp\mathcal B^c_r(f)+
\sharp\mathcal B^o_{r-1}(f)+\\
\sharp\bigl\{(\lambda,k)\in\overline{\mathcal J}_{r}(f)\bigm|\lambda(J)=1\bigr\}+\\
\sharp\bigl\{(\lambda,k)\in\overline{\mathcal J}_{r-1}(f)\bigm|\lambda(J)=1\bigr\}
\end{cases}
\\
\hskip .4in c.&\dim\im\bigl(H_r(X_\theta)\to H_r(X)\bigr)
&&=\sharp\bigl\{I\in \mathcal B^c_r(f)\bigm|\theta\in I\bigr\}
\end{align*}
where $n_{\theta}(I)=\sharp\{k\in\mathbb Z\mid\theta+2\pi k\in I\}$ and for $J=(V,T)$, $k(J)=\dim V$. 
\vskip .1in

2.  If $f\colon X\to\mathbb R$ is a tame map\footnote{A real valued map can be considered an angle valued by identifying $\mathbb R$ with $S^1\setminus 1$.} then:
\begin{align*}
a. &  \dim H_r(X_t)&&=\sharp\bigl\{I\in\mathcal B_r (f)\bigm|I\ni t\bigr\}
\\ 
b. &  \dim H_r(X)&&=\sharp\mathcal B^c_r(f)+\sharp\mathcal B^o_{r-1}(f)
\\
c. &  \dim\im\bigl(H_r(X_t)\to H_r(X)\bigr)
&&=\sharp\bigl\{I\in\mathcal B^c_r(f)\bigm| I \ni t\bigr\}.
\end{align*}
\end{theorem}
\vskip .2in


Consider $\xi_f\in H^1(X;\mathbb Z)$ the cohomology class represented by $f$ and for any $u\in\kappa\setminus 0$ denote by $u\xi_f$ the rank one representation  
\begin{equation}\label{E111}
u\xi_f\colon H_1(M;\mathbb Z)\to\mathbb Z\to\kappa\setminus0=\GL_1(\kappa)
\end{equation} 
with the last arrow given by ${n\to u^n}$.  Denote by $H_r(X;u\xi_f)$ the $r$-homology with coefficients in this representation
which is a $\kappa$-vector space. 
Theorem~\ref{T1} can be extended to the following theorem.

\begin{theorem}\label{T2}  If $f\colon X\to S^1$ is a tame map then:

1.  
\begin{equation*}
\dim H_r(X; u\xi_f)=\begin{cases} 
\sharp\mathcal B^c_r(f)+ 
\sharp\mathcal B^o_{r-1}(f)+\\ 
\sharp\{J\in \overline{\mathcal J}_{r}(f)| u\lambda(J)=1\}+\\
\sharp\{J\in\overline{\mathcal J}_{r-1}(f)|u^{-1}\lambda(J)=1\}.
\end{cases}
\end{equation*}

2. 
\begin{equation*} 
N_r(X;\xi_f)=
\sharp\mathcal B^c_r(f)+\sharp\mathcal B^o_{r-1}(f).
\end{equation*}
\end{theorem}

\vskip .2in

Denote by:
$\tilde f\colon\tilde X\to\mathbb R$ the infinite cyclic cover of $f\colon X\to S^1$.
\[
\begin{CD}\label{D1}
\tilde X                           @>\tilde f>> \mathbb R\\
@V\psi VV                                                       @Vp VV\\
X @>f >>            \mathbb S^1
\end{CD}
\]

Let $\tilde X_{[a,b]}:=\tilde f^{-1}{[a,b]}$, $\tilde X_t=f^{-1}(t)$. Clearly one has $\tilde X_t=X_{p(t)}$.


Denote by: 
\begin{align*}
\tilde{\mathcal B}_r(f)&=\bigl\{ I+2\pi k\bigm|k\in \mathbb Z, I\in \mathcal B_r(f)\bigr\},
\\ 
\tilde{\mathcal B}^c_r(f)&=\bigl\{ I+2\pi k\bigm|k\in \mathbb Z, I\in \mathcal B^c_r (f)\bigr\},
\\
\tilde{\mathcal B}^o_r(f)&=\bigl\{ I+2\pi k\bigm|k\in \mathbb Z, I\in \mathcal B^o_r (f)\bigr\}.
\end{align*}

We have

\begin{theorem}\label {T3}    
 If $f\colon X\to S^1$ is a tame map then:

1.
\begin{align*}
a. & \dim H_r(\tilde X_{[a,b]})
&&=\begin{cases} 
\sharp\{I\in \tilde{\mathcal B}_r(f), I\cap [a,b] \ne\emptyset \}+\\
\sharp\{I\in \tilde{\mathcal B}^o_{r-1}(f), I\subset [a,b]\}+\\ 
\sum_{J\in \mathcal J_r(f)}k(J).
\end{cases}
\\
b.& \dim\im\bigl(H_r(\tilde X_{[a,b]})\to H_r(\tX)\bigr) 
&&=\begin{cases} 
\sharp\{I \in\tilde{\mathcal B}^c_r(f), I\cap [a,b] \ne\emptyset\}+\\
\sharp\{I\in \tilde{\mathcal B}^o_{r-1}(f), I\subset [a,b]\}+\\ 
\sum_{J\in \mathcal J_r(f)} k(J).
\end{cases}
\\
c.& \dim \im\bigl(H_r(\tilde X_{[a,b]}) \to H_r(X)\bigr) 
&&=\begin{cases}
\sharp\{I\in \mathcal B_r^c,[a,b]\cap (I+2\pi k)\ne 0\}+\\
\sharp\{I\in \mathcal B^o_{r-1}\mid   I+2\pi k \subset [a,b]\}+\\
\sharp\{J\in \overline{\mathcal J}_{r}(f)| \lambda(J)=1\}
\end{cases}
\end{align*}

2.  
$V_r(\xi_f):=\ker (H_r(\tilde X)\to H^N_r(X;\xi_f))$ is a finite dimensional $\kappa$-vector space and $(V_r(\xi_f), T_r(\xi_f))=(V_r(f), T_r(f))$ 

3. 
$H_r(\tilde X)=\kappa[T^{-1},T]^N\oplus V_r(\xi_f)$ as $\kappa[T^{-1}, T]$-modules with $N=N_r(f)=\sharp{\mathcal B^c_r(f)} + \sharp{\mathcal B^o_{r-1}(f)}$.
\end{theorem}

\begin{obs}
Theorem \ref{T3} (1)  remains true if one replaces  a  closed interval by  a finite union of closed intervals  (possibly points).
\end{obs}


As a consequence we have the main result of this paper:

\begin{corollary}\label{T4} 
If $f_1,f_2\colon X_i\to S^1$ are two homotopy equivalent tame maps, then: 

1. $\sharp\mathcal B^c_r (f_1)+\sharp\mathcal B^o_{r-1}(f_1)=\sharp\mathcal B^c_r(f_2)+\sharp\mathcal B^o_{r-1}(f_2)$.

2. $\mathcal J_r(f_1)=\mathcal J_r(f_2)$. 
\end{corollary}

One can provide an alternative geometric description of the equivalence class of pairs $(V_r(f),T_r(f))$.    
Start with the tame map $X\to S^1$ representing the cohomology class $\xi\in H^1(X;\mathbb Z)$ and choose 
an angle $\theta$. Consider the compact space $X^{f,\theta}$ the cut of $X$ along $X_\theta$.  Precisely 
as a set this is the disjoint union $X^-_\theta \sqcup (X\setminus X_\theta) \sqcup X^+_\theta$ with the 
$X^\pm_\theta$  copies of $X_\theta$. The topology of $X^{f,\theta}$ is the obvious topology.\footnote{This is the unique topology which induces on 
$X\setminus X_\theta$, $X^\pm_\theta$ the same topology as $X$ and makes of $f^{-1}((\theta-\epsilon, \theta])$ resp.\ $f^{-1}([\theta,\theta+\epsilon)$
for $\epsilon$ small, neighborhoods of $X^-_\theta$ resp.\ $X_\theta^+$ in $X^{f,\theta}$.}

We have the two inclusions $i^-\colon X_\theta\to X^{f,\theta}$ and $i^+\colon X_\theta\to X^{f,\theta}$, which induce the linear maps 
$(i^-)_r\colon H_r(X_\theta)\to H_r(X^{f,\theta})$ resp.\ $(i^+)_r\colon H_r(X_\theta)\to H_r(X^{f,\theta})$.
These two linear maps define the linear relation $\mathcal R^\theta_r\subset H_r(X^-_\theta)\oplus H_r(X^+_\theta)$ defined by 
$(i^-)_r(x^-)=(i^+)_r (x^+)$, $x^\pm\in H_r(X^\pm_\theta)$, cf.~section~\ref{S6} for definitions.
\vskip .2in 

\begin{theorem}\label{P1} 
The regular part \footnote{ The regular part of a linear relation $\mathcal R\subset W\times W$ is a sub relation $\mathcal R^{reg}\subset V\times V$ 
which is given by an isomorphism $T\colon V\to V$ and is maximal, cf.~section~\ref{S6}.} of the relation $\mathcal R^\theta$ is isomorphic to $(V_r(f),T_r(f))$. 
\end{theorem}

As a consequence  of  Theorems~\ref{T2} and \ref{T4} and of the fact that any homotopy class contains tame maps we have:

\begin{corollary}\label {Cor}
1. If a tame angle valued map is homotopic to a fibration there are no closed and no open bar codes. 

2. For any tame map   $f: X\to S^1$ one has $\beta_r(X)-\mathcal J^1_r(f) - \mathcal J^1_{r-1}(f)=N_r(f)$.
\end{corollary}

\subsection{ Organizing the closed and open bar codes}

For $f \colon X\to\mathbb R$ resp.\ $f\colon X\to\mathbb S^1$ a tame map Theorem~\ref{T1}(4.) resp.\ Theorem~\ref{T2}(2.) suggest to collect together 
the closed $r$-bar codes and the open $(r-1)$-bar codes as configuration of points in $\mathbb R^2$ resp. $\mathbb T=\mathbb R^2/\mathbb Z.$ 
Precisely $\mathbb T$ is the quotient space of $\mathbb R^2$ the Euclidean plane, by the additive group of integers $\mathbb Z$, w.r.\ 
to the action $\mu: \mathbb Z\times \mathbb R^2 \to \mathbb R^2$ given by $\mu(n;(x,y))=(x+2\pi n,y+2\pi n).$

One denotes by $\Delta \subset \mathbb R^2$ resp.$\Delta \subset \mathbb T$ the diagonal of $\mathbb R^2$ resp. the quotient of the diagonal of $\mathbb R^2$ by the group $\mathbb Z$.
The points above or on diagonal, $(x,y)$, $x\leq y$,  will be used to record closed bar codes $[x,y]$ and 
the points below the diagonal, $(x,y)$, $x>y$, to record open bar codes $(y,x)$. This convention comes from the observation that continuous deformation of tame maps can produce deformation of an $r$-closed bar code $[x, y]$ into an $(r-1)$-open bar code $(y', x')$ but not without passing through a closed bar codes  with equal ends (located on $\Delta$) $[x'', y''],\ \ x''=y''.$ 


One can identify $\mathbb T$ with $\mathbb C\setminus 0$ sending the point of $\mathbb T$ represented by the pair $(x,y)$ to $e^{(y-x) +  i x}\in \mathbb C$.  
Clearly, $\Delta$ became the circle of radius $1$.

For an integer $k$ and $X$ a space, in our case $X=\mathbb T$ or $\mathbb R^2$, denote by $S^k(X)$ the $k$-th symmetric power of $X$, i.e.\ the quotient space 
$X^n/\Sigma_n$ with $\Sigma_n$ the symmetric group acting on $X^n$ by permutations and $X^n=\underbrace{X\times\cdots\times X}_n$.
 

In view of Theorem~\ref{T1}(4.) resp.\ Theorem~\ref{T2}(2.) for a tame real resp. angle valued map $f$  and any $r$ 
 we will collect the closed $r$-barcodes and the open $(r-1)$-bar codes as a point $C_r(f)\in S^{\beta_r(X)}(\mathbb R^2)$ where $\beta_r(X)=\dim H_r(X)$ resp.\ as a point 
$Cr(f) \in S^{\N_r(X; \xi_f)} (\mathbb T)$.
 If we identify a point in $(x,y)\in R^2$ with $z=x+iy$ it is convenient to regard $C_r(f)$ as 
the monic polynomial $P^f_r(z)$ of degree $\beta_r(X)$ whose roots are the elements of $C_r(f)$.  
Similarly, using the identification of $\mathbb T$ with $\mathbb C\setminus 0$  it is convenient to regard  
$C_r(f)$ as a monic  polynomial of degree $N_r(X;\xi_f)$. 
\vskip .1in
 
For a generic set  of continuous maps $f\colon X\to \mathbb C\setminus 0$  both $|f|\colon X\to \mathbb R$ and $f/|f|\colon X\to S^1$ 
are tame and consequently one obtains for any $r$ the pair of polynomials $(P_r(f),P_r(f/|f|))$ which can be viewed as refinements 
of Betti numbers of $X$ and of Novikov--Betti numbers of $(X;\xi_{f/|f|})$.

\vskip .1in

One expects that the assignment $f \rightsquigarrow C_r(f)$ defined on the space $T(X;\mathbb R)$ resp.\ $T(X;\mathbb S^1)$ of tame maps be continuous  
with respect to the compact open topology and extends by continuity to $C(X;\mathbb R)$ resp. $C(X; S^1).$  In particular one expects that the closed and open bar codes 
as read off $Cr(f)$  be defined for any continuous map $f\colon X\to\mathbb R$ resp. $f\colon X\to S^1,$  hence the  polynomials  considered  above for tame maps 
can be defined for any continuous map $f$  whose source and levels are compact ANR's and depend continuously on $f$.  This will be shown to be true in \cite{B12}. 

\section{Graphs representations}\label{S2}

In this section we summarize known facts about the representations of two graphs, $\mathcal Z$ and $G_{2m}$ and formulate some technical 
used in the proof of Theorems~\ref{T1}, \ref{T2}, \ref{T3}.

We consider two oriented graphs, $\Gamma=\mathcal Z$ whose vertices are $x_i$, $i\in\mathbb Z$, 
and arrows $a_i\colon x_{2i-1}\to x_{2i}$ and $b_i\colon x_{2i+1}\to x_{2i}$ 
\begin{center}
$$
\xymatrix{
\cdots & x_{2i-1}\ar[l]_-{b_{i-1}} \ar[r]^-{a_i} & x_{2i} & x_{2i+1}\ar[l]_-{b_{i}} \ar[r]^-{a_{i+1}} & x_{2i+2} & \cdots \ar[l]_-{b_{i+1}}
}
$$
The graph $\mathcal Z$
\end{center}
and $\Gamma=G_{2m}$ whose vertices are $x_1,x_2,\dotsc,x_{2m}$ and arrows $a_i$, $1\leq i\leq m$, and $b_i$, $1\leq i\leq m-1$, as above and $b_M\colon x_1\to x_{2m}$.  
\begin{center}
$$
\xymatrix{ &x_2\\
x_3\ar[ur]_{b_1}\ar[d]^{a_2}& & x_1\ar[ul]^{a_1}\ar[d]_{b_m}\\
x_4& & x_{2m}\\
x_{2m-3} \ar@{<.>}[u]\ar[dr]^{a_{m-1}}& & x_{2m-1}\ar[u]^{a_m}\ar[dl]_{b_{m-1}}\\
& x_{2m-2}
}
$$
The graph $G_{2m}$
\end{center}

Let $\kappa$ a fixed field. 

A $\Gamma$-representation $\rho$ is an assignment which to each vertex $x$ of $\Gamma$ assigns a finite dimensional vector space $V_x$  
and to each oriented  arrow  from the vertex $x$ to the vertex $y$ a linear map $V_x\to V_y$. 
The concepts of morphism, isomorphism= equivalence, sum, direct summand, zero and nontrivial representations are obvious.

A $\mathcal Z$-representation is given by the collection 
\begin{equation*}
\rho:=
\begin {cases}
\begin {aligned}  
V_r,\quad \alpha_i:V_{2i-1}\to &V_{2i}, \quad \beta_i:V_{2i+1}\to V_{2i}\\ 
r, i \in &\mathbb Z  \ \ ,
\end{aligned}
\end{cases}
\end{equation*}
abbreviated to $\rho=\{V_r,\alpha_i,\beta_i\}$, while a $G_{2m}$ representation by the collection 
\begin{equation*}
\rho:=
\begin {cases}
\begin {aligned}  
V_r,\quad \alpha_i:V_{2i-1}\to V_{2i}, \quad \beta_i:V_{2i+1}\to V_{2i} \\
1\leq r \leq 2m, \quad 1\leq i\leq m, \quad V_{2m+1}= V_1
\end{aligned}
\end{cases}
\end{equation*}
also abbreviated to $\rho=\{V_r,\alpha_i,\beta_i \}.$

A representation $\rho$ is \emph{regular} if all the linear maps $\alpha_i$ and $\beta_i$ are isomorphisms.  

Any regular $G_{2m}$-representation $\rho=\{V_r,\alpha_i,\beta_i\}$ is equivalent to the representation 
\begin{equation}\label{R1}
\rho(V,T) = \{V'_r=V, \alpha'_1= T, \alpha'_i=Id\   i\ne 1, \  \beta'_i= Id\}
\end{equation}
with $T=\beta^{-1}_m\cdot \alpha_m^{-1}\cdots \beta_1^{-1}\cdot\alpha_1$ \footnote{ 
The isomorphism is provided by the linear maps $\omega_r\colon V_r\to V_r$ given by 
\begin{equation*}
\begin{cases}
\begin{aligned}
\omega_1= &Id\\
\omega_2= &\beta_m^{-1}\cdots \beta^{-1}_2\cdots \alpha_2\cdot \beta_1^{_1}\\
\omega_3= &\beta_m^{-1}\cdots \beta^{-1}_2\cdots \alpha_2\\
&\cdots\\
\omega_{2m}=& \beta^{-1}
\end{aligned}
\end{cases}
\end{equation*}.}.

A $\mathcal Z$-representation $\rho$ has \emph{finite support} if $V_i=0$ for all but finitely many $i$.  
There are no nontrivial regular $\mathcal Z$-representations with finite support.

For $\mathcal Z$-representation $\rho=\{V_r,\alpha_i,\beta_i\}$ we denote by $T_{k,l}(\rho)$, $k\leq l$ the representation with finite support 
$T_{i,j}(\rho)=\{V'_r,\alpha'_i,\beta'_i\}$ defined by 
\begin{equation}
\begin{aligned}
V'_r=&\begin{cases} V_r  \ \  2k\leq r\leq 2l\\
 0 \ \ \rm{otherwise}
\end{cases}\\
\alpha'_r=&\begin{cases} \alpha_r  \ \  k+1\leq i\leq l\\
 0 \ \ \rm{otherwise}
\end{cases}\\
\beta'_r=&\begin{cases} \beta_r  \ \  k\leq r\leq l-1\\
 0 \ \ \rm{otherwise}.
\end{cases}
\end{aligned}
\end{equation}

A  representation $\rho$ is \emph{indecomposable} 
if not the sum of two nontrivial representations. It is well known and not 
hard to prove that any $\mathcal Z$-representation with finite support and any $G_{2m}$-representation  
can be uniquely decomposed in a finite sum of indecomposable representations (the Remack--Schmidt theorem) and 
these indecomposables are unique up to isomorphism cf.~\cite{HDJW}.


{\bf The indecomposable $\mathcal Z-$representations with finite support}  are indexed by four type of intervals  $I$ with ends $i$ and $j$ and denoted  by $\rho(I)$ or more precisely by:

1. $ \rho([i,j])  $,   \quad    2. $ \rho([i,j))  $,   \quad   3.  $ \rho((i,j])  $ and \quad  4. $  \rho(i,j) $ 

\noindent with  $ i\leq j$ in case (1.) and  $ i<j $  for the cases (2., 3., 4.) above.  They 
have all vector spaces  either one dimensional or zero dimensional  and the linear maps $\alpha_i, \beta_j$  the identity if both the source and the target are nontrivial and zero otherwise.
Both the indexing interval $I$ and the representation $\rho(I)$ will be called {\it bar codes}.

Precisely, 

\begin{enumerate}
\item 
$ \rho([i,j]) , i\leq j$  has 
$V_r= \kappa$ for $r = \{2i, 2i+1, \cdots 2j\}$  and $V_r=0$ if $r\ne [2i,2j]$
\item
$ \rho([i,j)) , i< j $ has
$V_r= \kappa$ for $r = \{2i, 2i+1, \cdots 2j\}$  and $V_r=0$ if $r\ne [2i,2j-1]$
\item 
$ \rho((i,j]) , i< j $ has
$V_r= \kappa$ for $r = \{2i, 2i+1, \cdots 2j\}$   and $V_r=0$ if $r\ne [2i+1,2j]$
\item 
$ \rho((i,)]) , i< j $ has
$V_r= \kappa$ for $r = \{2i, 2i+1, \cdots 2j\}$  and $V_r=0$ if $r\ne [2i+1,2j-1]$
\end{enumerate}
 with all $\alpha_i$ and $\beta_i$ the identity  provided that the source and the target are  both non zero.

The above description is implicit in \cite{G72}.

Denote by $\mathcal B(\rho)$ the collection of bar codes which appear as direct summands of $\rho$, and by  $\mathcal B^c(\rho),$ resp. $\mathcal B^o(\rho),$ resp, $\mathcal B^m(\rho)$ the subsets of $\mathcal B(\rho)$ consisting of bar codes with both ends closed, resp. open resp. one open one closed.
By Remack- Schmidt theorem any $\mathcal Z-$ representation $\rho$ can be uniquely written as 
\begin{equation}\label{E3}
\rho= \sum_{I\in\mathcal B(\rho)}  \rho(I).
\end{equation}
\vskip .2in
{\bf The indecomposable $G_{2m}-$ representations } are of two types, type I and type II.

{\bf Type I:} ({\it bar codes})
For any triple of integers $\{i,j,k\}$, $1 \leq i,j \leq m,$ $k\geq 0$, we have the representations  denoted by 

\begin{enumerate}
\item 
$\rho^I([i, j];k) \equiv \rho^I([i, j+mk]),\quad 1\leq i, j\leq m, k\geq 0$
\item  $\rho^I((i,j];k) \equiv  \rho^I((i, j+mk]),\quad 1\leq i, j\leq m, k\geq 0$ 
\item  $\rho^I([i,j);k) \equiv  \rho^I([i, j+mk)),\quad 1\leq i, j\leq m, k\geq 0$ 
\item $\rho^I((i,j);k)  \equiv  \rho^I((i, j+mk)),\quad 1\leq i, j\leq m, k\geq 0$ 
\end{enumerate}
 described as follows.

Suppose the vertices  of $G_{2m}$ are located counter-clockwise on the unit circle
with  evenly indexed vertices $\{x_2, x_4, \cdots x_{2m}\}$ 
 corresponding  to the angles 
$0< s_1< s_2 <\cdots <s_{m} \leq2\pi.$ 
Draw the spiral curve 
 for $a= s_i$ and $b=s_j+2\pi k$ with the ends a black or an empty circle if the end is closed or open 
(see picture below for $k=2$).

\begin{figure}[h]
\includegraphics [height=4.5cm]{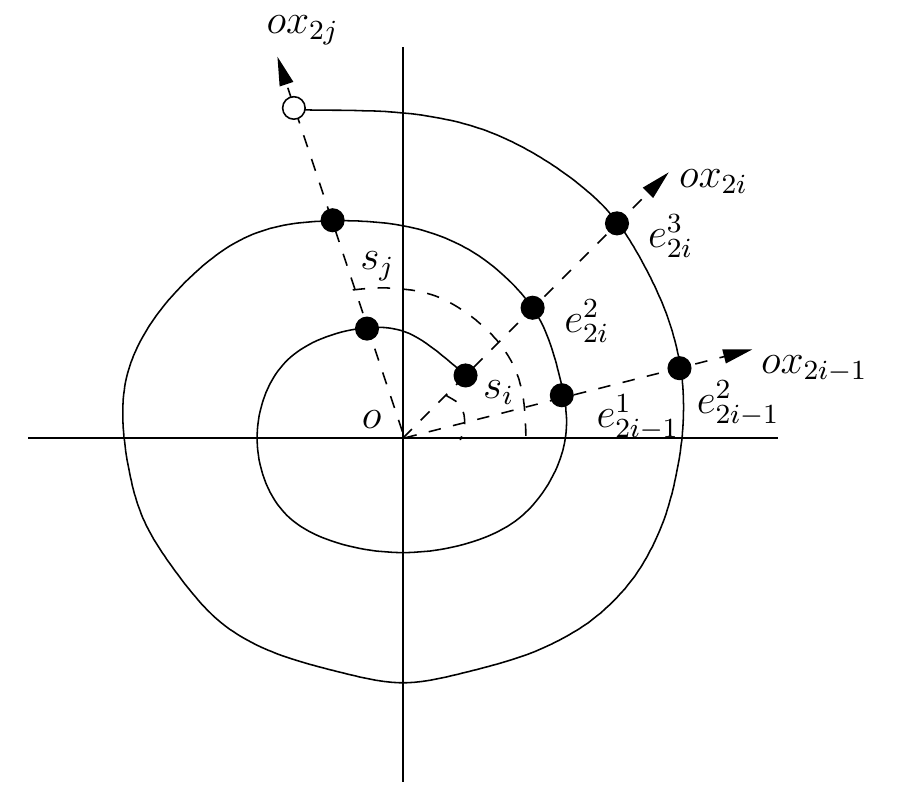}
\caption{The spiral for $[i, j+2m)$.}
\label{barcode}
\end{figure}

Denote by $V_i$ the vector space generated by the intersection points of the spiral with the radius corresponding to the vertex $x_i$ and let $\alpha_i$ resp.$\beta_i$ be defined on bases  in an obvious manner ; a generator $e$ of $V_{2i\pm1}$ is sent to the generator $e'$ of $V_{2i}$ if connected by a piece of spiral  and to $0$ otherwise.
\vskip .2in

{\bf Type II: } The representations  of  Type II  are regular representations associated to a Jordan block 
$J=(V, T)$  cf. formula (\ref{R1}) and denoted by $\rho^{II}(J)$.  They are clearly indecomposable.

In consistency with the above conventions we refer to both $J=(V,T)$ and the representation $\rho^{II}(J)$ as {\it Jordan block}.  
If the eigenvalues of $T$ are in $\kappa,$ in particular if $\kappa$ is algebraically closed, $J= (V,T)$ is indecomposable iff $T$ is conjugate to $T(\lambda, k)$ defined by formula  (\ref {E1}) for some $\lambda\in \kappa\setminus 0$ and in this case  
we will write $\rho^{II}(\lambda, k)$ for the representation  $\rho^{II}(\kappa^k, T(\lambda,k)).$ In consistency with the above convention  we refer to both, the representation $\rho^{II}(\lambda, k)$ and  the pair  $(\lambda, k)$ as {\it Jordan cell}. 
 If $\kappa$ is not algebraically closed and $(V,T)$ is a Jordan block  then $(V\otimes \overline\kappa, T\otimes \overline\kappa) $, $\overline \kappa$ the algebraic closure of $\kappa,$ does not necessary remain a Jordan block. However it decomposes uniquely as a finite sum  of Jordan cells. Two Jordan blocks are equivalent iff they remain equivalent after tensored by $\overline\kappa,$ equivalently the associated Jordan cells over $\overline\kappa$ are the same.  

By Remack-Schmidt theorem  
any $G_{2m}-$representation $\rho$ can be uniquely decomposed as 

\begin{equation}\label{E4}
\begin{aligned}
\rho= \bigoplus_{I\in \mathcal B(\rho)} \rho^I(I)\oplus \bigoplus_{J\in\mathcal J(\rho)} \rho^{II}(J).
\end{aligned}
\end{equation}
The above description is implicit in \cite {N73} and \cite {DF73}.

Introduce 
$$\rho_{reg}= \bigoplus_{J\in\mathcal J(\rho)} \rho^{II}(J)$$
\noindent with $\rho_{reg}= \rho(V_{reg}(\rho), T_{reg}(\rho)).$
The pair $(V_{reg}(\rho), T_{reg}(\rho))$ is also referred to as the {\it monodromy}  of $\rho.$ 

\vskip .1in

In \cite{BD11} an algorithm to provide the decomposition of a $G_{2m}-$representation as a sum of indecomposable is described. The algorithm holds for $\mathcal Z-$representations  too and is based on four elementary transformations $T_1(i), T_2(i), T_3(i), T_4(i)$ described for the reader convenience in the Appendix.  They will be used  in section \ref{S7}. 
Performing any of  these transformations one passes from a representation $\rho$ to a representation $\rho'$  of  strictly smaller dimension (of the total vector space  $\oplus _i V_i$ or $\oplus_{1\leq i\leq 2m} V_i$), with  the same monodromy (in case of $G_{2m}-$representation) and with bar codes changed in a specified way. After applying such transformations  finitely many time one ends up with a regular representation 
and  by  backward  book keeping, one can reconstruct the initial collection of bar codes too. 
\vskip .2in

\vskip .1in

To a $\mathcal Z-$ representation $\rho= \{V_r, \alpha_i, \beta_i\} $ $ r,i\in \mathbb Z$ one associates the linear transformation $M(\rho): \oplus  V_{2i-1}\to \oplus V_{2i}$  given by the infinite block matrix  with entries 

\begin{equation}
M(\rho)_{2r-1, 2s}=
 \begin{cases}
&\alpha_r, \ \ \ \rm { if } \ \ s=r\\
& \beta_{r-1}, \rm {if} \ s=r-1\\
& 0 \ \ \ \ \ \ \rm {otherwise} \ .
\end{cases}
\end{equation}

To 
a $G_{2m}-$ representation  $\rho=\{ V_r, \alpha_i, \beta_i\}$ $1\leq r\leq 2m, 1\leq i\leq m.$
 one  associates the block  matrix 
$M(\rho):  \bigoplus_{1\leq i\leq m} V_{2i- 1} \to   
\bigoplus _{1\leq i\leq m}V_{2i}$ 
defined by: 

\begin{equation*}
\begin{pmatrix}

\alpha_1& -\beta_1 &      0       &\dots  & \dots &0\\
         0    & \alpha_2 &-\beta_2&\dots   &  \dots &0\\
     \vdots&\vdots       &\vdots     &\vdots &\vdots & \\
      0        &\dots        &\dots       &\dots    &\dots \alpha_{m-1}&-\beta_{m-1}\\
      -\beta_m&\dots    &\dots       &\dots   & \dots &\alpha_m
         \end{pmatrix} _.
         \end{equation*}
\vskip .1in
For a $\Gamma= \mathcal Z \ \rm {or}\  G_{2m}-$ representation $\rho$ denote  by:
\begin{enumerate}
 \item  $\dim(\rho): \Gamma \to \mathbb Z_{\geq 0}$ the function
  defined by $r \rightsquigarrow  \dim(V_r)$ 

 \item $n_i:=\dim(V_{2i-1})$ and $r_i:= \dim (V_{2i}).$ 

 \item $d\ker (\rho)= \dim \ker M(\rho)$ and 

\item $d\coker (\rho)= \dim \coker M(\rho).$ 
\end{enumerate}
For a $G_{2m}-$ representation $\rho=\{V_r, \alpha_i,\beta_i\}$ and $u\in \kappa \setminus 0$ denote by  $\rho_u=\{V'_r, \alpha'_i, \beta'_i\}$  the representation with $V'_r= V_r,$
$\alpha'_1= u \alpha_1,$ $\alpha'_i=\alpha_i$ for $i\ne1$ and $\beta'_i=\beta_i.$ 
 Clearly  $(\rho_1 \oplus \rho_2)_u= (\rho_1)_u\oplus (\rho_2)_u,$ $\dim(\rho)= \dim(\rho_u)$
and the block matrix   $M(\rho_u)$ is given by 
\begin{equation*}
\begin{pmatrix}

u\alpha_1& -\beta_1 &      0       &\dots  & \dots &0\\
         0    & \alpha_2 &-\beta_2&\dots   &  \dots &0\\
     \vdots&\vdots       &\vdots     &\vdots &\vdots & \\
      0        &\dots        &\dots       &\dots    &\dots \alpha_{m-1}&-\beta_{m-1}\\
      -\beta_m&\dots    &\dots       &\dots   & \dots &\alpha_m
         \end{pmatrix} _.
         \end{equation*}

One has:

\begin{proposition}$($\cite{BD11}$)$ \label {PR}
~
\begin{enumerate}
\item $\dim(\rho_1\oplus \rho_2)= \dim(\rho_1) + \dim(\rho_2)$,  
\item $d\ker (\rho_1\oplus \rho_2)= d\ker (\rho_1) + d\ker (\rho_2)$,
\item $d\coker(\rho_1\oplus \rho_2)= d\coker(\rho_1) + 
d\coker (\rho_2),$
\item 
$d\ker (\rho)=d\ker(\rho_u), \  d\coker(\rho)=d\coker(\rho_u)$. 
\end{enumerate}
\end{proposition}
 
 For  the indecomposable $\mathcal Z-$ representations one has: 

\begin{proposition}\label {AP01}
~
\begin {enumerate}
\item 
\begin{equation*}
\dim {\rho}([i,j])= \begin{cases}
n_l=1 , i+1\leq l\leq j, \ =0 \ \rm otherwise        \\
r_l=0  , i\leq l \leq j, \ =0 \ \rm otherwise
\end{cases}
\end{equation*}

\item
\begin{equation*}
\dim {\rho}((i,j))= \begin{cases}
n_l=1 , i+1\leq l\leq j, \ =0 \ \rm otherwise        \\
r_l=0  , i+1\leq l \leq j-1, \ =0 \ \rm otherwise
\end{cases}
\end{equation*}

\item
\begin{equation*}
\dim {\rho}([i,j))= \begin{cases}
n_l=1 , i+1\leq l\leq j, \ =0 \ \rm otherwise        \\
r_l=0  , i\leq l \leq j-1, \ =0 \ \rm otherwise
\end{cases}
\end{equation*}

\item
\begin{equation*}
\dim {\rho}((i,j])= \begin{cases}
n_l=1 , i+1\leq l\leq j, \ =0 \ \rm otherwise        \\
r_l=0  , i+1\leq l \leq j, \ =0 \ \rm otherwise
\end{cases}
\end{equation*}
\end{enumerate}
\end{proposition}

\begin{proposition}\label{AP02}
~
\begin{enumerate} 
\item $d\ker  \rho ([i,j])=0$, $d\coker \rho ([i,j])=1$,
\item $d\ker \rho ([i,j))=0$, $d\coker \rho ([i,j))=0$,
\item $d\ker  \rho ((i,j])=0$, $d\coker \rho ((i,j])=0$,
\item $d\ker  \rho ((i,j))=1$, $d\coker \rho ((i,j))=0$.

\end{enumerate}
\end{proposition}

For indecomposable $G_{2m}-$representations one has:

\begin{proposition}   $($\cite{BD11}$)$\label {AP1} 
~
\begin{enumerate}

\item If $i\leq j$ then 
\begin{enumerate}
\item $\dim \rho^I([i,j]; k)$ is given by: 

$n_l=k+1$ if  $(i+1)\leq l\leq j$ and $k$ otherwise, 

$ r_l= k+1$ if $i\leq l\leq j$ and $k$ otherwise 
\item $\dim \rho^I((i,j]; k)$ is given by: 

$n_l=k+1$ if  $(i+1)\leq l\leq j$ and  $k$ otherwise, 

$r_l= k+1$ if $(i +1)\leq l\leq j$ and $k$ otherwise, 
\item $\dim \rho^I([i,j); k)$ is given by: 

$n_l=k+1$ if  

$(i+1)\leq l\leq j$ and $k$ otherwise,  

$ r_l= k+1$   if  $i\leq l\leq (j-1)$ and $k$ otherwise,   
\item $\dim \rho^I((i,j); k)$ is given by:

 $n_l=k+1$ if  
$(i+1)\leq l\leq j$ and $k$ otherwise,  

$r_l= k+1$  if  $(i +1)\leq l\leq( j-1)$ and $k$ otherwise

\end{enumerate}
\item If $i> j$ then similar statements hold.
\begin{enumerate}
\item $\dim \rho^I([i,j]; k)$ is given by:

 $n_l=k$ if  $(j+1)\leq l\leq i$ and $k+1$ otherwise; 
 
 $r_l=k$ if   $(j+1)\leq l\leq (i-1)j$ and  $k+1$ otherwise
\item $\dim \rho^I((i,j]; k)$ is given by:

 $n_l=k$ if  $(j+1)\leq l\leq i$ and  
 $k+1 $ otherwise.  
 
 $r_l=k$ if  
 $(j+1)\leq l\leq i$ and $k +1$ otherwise,
\item $\dim \rho^I([i,j); k)$ is given by:

 $n_l=k$ if  
$(j+1)\leq l\leq i$ and  $k+1$ otherwise;

 $r_l=k$ if  
$ j\leq l \leq (i-1)$  and $k+1 $ otherwise,
\item $\dim \rho^I((i,j); k)$ is given by:

 $n_l=k$  if $(j+1)\leq l\leq i$  and $k+1$ otherwise; 
 
 $r_l=k$ if  
$j\leq l\leq i$  and $k+1$ otherwise.  
\end{enumerate}
\end{enumerate}
\end{proposition}

\begin{proposition}  $($\cite{BD11}$)$  \label{AP2}
~
\begin{enumerate} 
\item $d\ker \rho^I([i,j]; k)=0$, $d\coker \rho^I([i,j]; k)=1$,
\item $d\ker \rho^I([i,j); k)=0$, $d\coker \rho ^I([i,j); k)=0$,
\item $d\ker \rho^I ((i,j]; k)=0$, $d\coker \rho^I((i,j]; k)=0$,
\item $d\ker \rho^I((i,j); k)=1$, $d\coker \rho^I((i,j); k)=0$,
\item $d\ker \rho^{II}(\lambda,k) =0$ $(\rm{resp.} =1)$ 
if $\lambda {\ne 1}$ $(\rm{resp.} =1),$
\item $d\coker \rho^{II}(\lambda,k)=0$ $(\rm{resp.} =1)$ 
if $\lambda {\ne 1}$ $(\rm{resp.} =1).$
\end{enumerate}
\end{proposition}

The proof of Propositions \ref {PR} (1,2,3), \ref{AP01}, \ref{AP1}  are straightforward. Items (i- vi)  in Proposition \ref{AP2}  
 follow from the calculation of the kernel of $M(\rho)$ and from Proposition \ref{AP1} while Proposition \ref{AP02} can be viewed as a particular case of  Proposition  \ref{AP2}.
Proposition \ref{PR} (iv) has to be verified first for indecomposable  representations and then in view of  Proposition \ref{PR} the statements  hold for an arbitrary representation. 

The calculation of kernel of $M(\rho)$ for $\rho$ of Type I or II boils down to the description of the space of solutions of the linear system
\begin{equation*}
\begin{aligned}
\alpha_1(v_1)=&\beta_1(v_3)\\
\alpha_2(v_3)=&\beta_2(v_5)\\
&\cdots\\
\alpha_m(v_{2m-1})=&\beta_m(v_1)
\end{aligned}
\end{equation*}
which were explicitly described above. 

Proposition \ref{AP02} and \ref{AP2} can be refined.

For this purpose let us choose once for all for any open resp. closed interval $I$ 
an isomorphism between $\ker \rho(I) $ resp. $\coker \rho(I)$  and $\kappa$ and for any Jordan cell  $(1, k)$ an isomorphism between  $\ker\rho^{II}(1, k)$ resp. $\coker \rho^{II}(1,k)$ and  $\kappa.$

For a set $S$ let $\kappa[S]$ denote the vector space generated by $ S.$
Recall that for a representation $\rho$ we have denoted by $\mathcal B^c(\rho)$ the collection of closed bar codes and by $\mathcal B^o(\rho)$ the collection of 
open bar codes.  
The following propositions follows immediately   from Propositions \ref {PR}, \ref{AP02} and \ref{AP2}. 

\begin{proposition}\label{O37}
If for  a $\mathcal Z-$ representation with finite support  $\rho$  a decomposition of $ \rho= \sum_{I\in \mathcal B( \rho)}\rho(I)$ is given,  then Proposition \ref{AP01} provides 
 canonical isomorphisms $$\Psi^c:\kappa [\mathcal B^c(\rho)]\to \coker M( \rho)$$ and $$\Psi^o:\kappa [\mathcal B^o( \rho)]\to \ker M(\rho).$$
\end{proposition}  
\vskip .1in

Let us write   $\overline{\mathcal J}^\lambda$ for the collection of  Jordan cells cells whose eigenvalue is exactly $\lambda.$
We have: 

\begin{proposition} \label{O38}
If $\rho$ is a $G_{2m}$ representation 
Proposition \ref{AP2}  provides the canonical isomorphisms 
\begin{equation*}
\begin{aligned}
\Psi^c:\kappa &[\mathcal B^c(\rho)\sqcup \overline{\mathcal J}^1(\rho)]\to \coker M(\rho)\\
\Psi^o:\kappa &[\mathcal B^o(\rho)\sqcup \overline{\mathcal J}^1(\rho)]\to \ker M(\rho).
\end{aligned}
\end{equation*} 
More general for any $u\in \kappa\setminus 0$ it provides the canonical isomorphisms 
\begin{equation*}
\begin{aligned}
 \Psi^c:\kappa &[\mathcal B^c(\rho)\sqcup \overline{\mathcal J}^{(u^{-1})}(\rho)]\to \coker M(\rho_u))\\
 \Psi^o:\kappa &[\mathcal B^o(\rho)\sqcup \overline{\mathcal J}^{(u^{-1})}(\rho)]\to \ker M(\rho_u).
\end{aligned}
\end{equation*} 
\end{proposition}
\vskip .1in

For  $\rho=\{V_r, \alpha_i, \beta_i\}$ a $G_{2m}-$ representation, consider   the $\mathcal Z-$ representation $\tilde \rho:= \{V'_{2mk+r}= V_r, \alpha'_{mk+i}= \alpha_i, \beta'_{mk+i}=\beta_i\}$ and  
denote by:   

$\tilde{\mathcal B}(\rho):= \{I+2\pi k \mid k\in \mathbb Z, I\in \mathcal B \},$ \

$\tilde{\mathcal B}^c(\rho):= \{I+2\pi k \mid k\in \mathbb Z, I\in \mathcal B^c \},$ 

$\tilde{\mathcal B}^o(\rho):= \{I+2\pi k \mid k\in \mathbb Z, I\in \mathcal B^o\}.$  

Let $\tilde {\mathcal J}(\rho)$ be  the set which contains  $\dim(V)$ copies of $J$ for any Jordan block $J=(V,T)\in \mathcal J(\rho) ,$ equivalently $k$ copies of each Jordan cell $(\lambda, k)\in \overline {\mathcal J}(\rho).$ 


 In section \ref{S5} we will need the following observation. 
\begin{obs}\label{O39}
 \begin{equation}\label{E}
\begin{aligned}
\mathcal B(T_{i,j} (\tilde\rho))= &\{  I \in \tilde{\mathcal B_r}(\rho)   \mid    I \cap [2i,2j]\ne \emptyset  \} \sqcup \tilde {\mathcal J}(\rho)
\\
\mathcal B^c(T_{i,j} (\tilde\rho))=& \{ I \in \tilde{\mathcal B}^c_r(\rho) \mid I \cap [2i,2j]  \rm  { a \ closed\  nonempty\ interval} \}\sqcup \tilde {\mathcal J}(\rho)\\
\mathcal B^o(T_{i,j}(\tilde\rho))=& \{ I \in \tilde{\mathcal B}^o_r( \rho),   I\subset (2i,2j)).
\end{aligned}
\end{equation} 
\end{obs}

The above statement can be easily verified for representations of Type I and II  and then follows for arbitrary representations.


 
 %
\section {Appendix to Graph Representations }

{\bf The Elementary transformations.}\label{ET}

\vskip .1in

We discuss here only $G_{2m}$ representations since $\mathcal Z-$ representations with finite support can be viewed  as particular cases.
We convene that for $i>2m$ $V_i= V_{i-2m},$ $\alpha_i= \alpha_{i-2m}$ and $\beta_i=\beta_{i-2m}.$

Each transformation takes an index $i$ and a representation 
\[\rho=\{V_j| 1\le j\le 2m , \, \alpha_s: V_{2s-1}\to V_{2s}, \, 
\beta_s: V_{2s+1}\to V_{2s}| 1\le s\le m\}\] 
and produces a new representation 
$$\rho'=\{V'_j| 1\le j \le 2m ,\,  \alpha'_s: V'_{2s-1}\to V'_{2s}, \, \beta'_s: V'_{2s+1}\to V'_{2s}| 1\le s\le m\}$$ 
as follows:
\begin{enumerate}
\item 
If $\rho'= T_1(i)\rho$  then  $V'_{2i-1}= V_{2i-1}/ \ker (\beta_{i-1})$,  $V'_{2i}=V_{2i}/\alpha_{i}(\ker(\beta_{i-1})),$  $V'_{j}= V_j$  for $j\ne \{2i-1, 2i\}$ with $\alpha'_s, \beta'_s$ being induced from $\alpha_s, \beta_s$ for
$s\in [1,m]$.
\vskip .1in
\item 
If  $\rho'= T_2(i)\rho$ then  $V'_{2i+1}= V_{2i+1}/ \ker (\alpha_{i+1})$,  $V'_{2i}=V_{2i}/\beta_{i}(\ker\alpha_{i+1}),$  $V'_{j}= V_j$  for $j\ne \{2i+1, 2i\}$ with $\alpha'_s, \beta'_s$ being induced from $\alpha_s, \beta_s$ for 
$s\in[1,m]$.
\vskip .1in
\item 
If $\rho'= T_3(i)\rho$  then $V'_{2i}= \alpha_i(V_{2i-1})$, $V'_{2i+1}=\beta_i^{-1}(\alpha_i(V_{2i-1})),$  $V'_{j}= V_j$ for $j\ne \{2i, 2i+1\}$ with $\alpha'_s, \beta'_s$ being the restrictions  of $\alpha_s, \beta_s$ for $s\in[1,m]$.
\vskip .1in
\item 
If $\rho'= T_4(i)\rho$ then $V'_{2i}= \beta_i(V_{2i+1})$, $V'_{2i-1}=\alpha_i^{-1}(\beta_i(V_{2i+1}))$,  $V'_{j}= V_j$ for $j\ne \{2i, 2i-1\}$ with $\alpha'_s, \beta'_s$ being the restrictions  of $\alpha_s, \beta_s$ for $s\in [1,m]$.
\end{enumerate}

The following diagrams \footnote{in these diagrams $V_0= V_{2m} $ 
and $\beta_0= \beta_m$} indicate the constructions described above. 
The indices increase
from right to left to signify that the vector spaces are laid counterclockwise
with increasing indices around a quiver.\\

Transformation $T_1(i)\rho$:

\xymatrix{
&\cdots&V_{2i+1}\ar[l]_{\alpha_{i+1}}\ar [dr]_{\beta'_i}\ar[r]^{\beta_i}&V_{2i}\ar[d]              &V_{2i-1}\ar[l]_{\alpha_{i}}\ar[d]\ar[r]^{\beta_{i-1}}           &V_{2i-2}      &\cdots\ar[l]&\\
&           &                                                                                      &V'_{2i}                      &V'_{2i-1}\ar[l]^{\alpha'_i} \ar[ur]_{\beta'_{i-1}}        &                     &          &}

\vskip .2in 

\hskip 1in $V'_{2i-1}= V_{2i-1}/  \ker(\beta_{i-1}) \quad  V'_{2i}= V_{2i}/ \alpha_{i}(\ker(\beta_{i-1}))$

\vskip .2in 

Transformation $T_2(i)\rho$:

\xymatrix{
&\cdots\ar[r]^{\beta_{i+1}}&V_{2i+2}&V_{2i+1}\ar[l]_{\alpha_{i+1}} \ar[d]\ar[r]^{\beta_i}&V_{2i}\ar[d]&V_{2i-1}\ar[l]_{\alpha_i} \ar[dl]^{\alpha'_i}\ar[r]^{\beta_i}&\cdots&\\
&                                           &                &V'_{2i+1}\ar[r]^{\beta'_i} \ar [ul]^{\alpha'_{i+1}}                &V'_{2i}         &                                        &          &}

\vskip .2in

\hskip 1in $V'_{2i+1}= V_{2i+1}/ \ker(\alpha_{i+1}), \quad  V'_{2i}= V_{2i}/ \beta_i(\ker(\alpha_{i+1}))$

\vskip .2in

Transformation $T_3(i)\rho$:

\xymatrix{
&\cdots\ar[r]^{\beta_{i+1}} &V_{2i+2}  &V_{2i+1}\ar[l]_{\alpha_{i+1}}\ar[r] ^{\beta_i}             &V_{2i}                  &V_{2i-1}\ar[dl]^{\alpha'_i} \ar[l]_{\alpha_i}\ar[r]^{\beta_i}&\cdots&\\
&                    &                 &V'_{2i+1}\ar[ul]^{\alpha'_{i+1}}\ar[u]\ar[r]^{\beta'_i}   &V'_{2i}\ar[u]         &                                        &          &}
\vskip .2in 
\hskip 1in $V'_{2i}= \alpha_{i}(V_{2i-1}) \quad  V'_{2i+1}= \beta_i^{-1}(\alpha_i(V_{2i-1}))$
\vskip .2in 

Transformation $T_4(i)\rho$:

\xymatrix{
&\cdots&V_{2i+1}\ar[l]_{\alpha_{i+1}}\ar[dr]_{\beta'_{i}}\ar[r]^{\beta_i}&V_{2i}                      &V_{2i-1}\ar[l]_{\alpha_{i}}   \ar[r]^{\beta_{i-1}}           &V_{2i-2}      &\cdots\ar[l]&\\
&           &                               &V'_{2i}\ar[u]                      &V'_{2i-1}\ar[l]_{\alpha'_i} \ar[u]\ar[ur]_{\beta'_{i-1}}  &                     &          &}

\vskip .2in

\hskip 1in $V'_{2i}= \beta_i(V_{2i+1})  \quad  V'_{2i-1}= \alpha^{-1}_{i}(\beta_i(V_{2i+1})).$

\vskip .1in

The following observations follow straightforwardly from the
definitions.

$T_1(i)\rho$  eliminates all bar codes of the form 
$(i-1,i)$ and $(i-1,i],$ if the case, 
shrinks each bar code of the form $(i-1,k]$ and $(i-1,k),$ $k\geq (i+2),$ 
into bar codes $(i,k]$ and $(i,k)$ respectively
with the convention that $(i-1,k\}$ is $(m, m+k\}$ when $i=1$, and
leaves all  other barcodes and Jordan cells unchanged. If $\beta_{i-1}$ is injective then $T_1(i)\rho=\rho.$

$T_2(i)\rho$ eliminates all bar codes of the 
form $(i,i+1)$ and $[i,i+1),$if the case,
shrinks each bar code of the form $[l,i+1)$ and $(l,i+1),$  $l\leq i-1,$ 
into bar codes $[l,i)$ and $(l,i)$ respectively,
and leaves any other barcodes and Jordan cells unchanged. If $\alpha_{i+1}$ is injective then $T_2(i)\rho=\rho.$

Type $T_3(i)\rho$ 
eliminates all bar codes of the form $[i,i]$ and $[i,i+1),$if the case,
shrinks each bar code of the forms $[i, k)$ and $[i,k]$, $k\geq i+1,$ into 
the bar codes $[i+1,k)$ and $[i+1, k]$ respectively,
and leaves all  other  type of barcodes and Jordan cells unchanged. If $\alpha_{i}$ is surjective then $T_3(i)\rho=\rho.$

Type $T_4(i)\rho$ 
eliminates all bar codes of the form $[i,i]$ and $(i-1,i],$ if the case,
shrinks each bar code of the forms $(l, i] $ and $[l,i],$ $l\leq i-1,$ 
into the bar codes $(l, i-1]$ and $[l,i-1]$ respectively 
with the convention that $\{l,0\}$ is identified to $\{l+m, m\}$, 
and leaves all other type of barcodes and Jordan cell unchanged.If $\beta_{i}$ is surjective then $T_4(i)\rho=\rho.$

In deciding "if the case " the following proposition is of use.
Let $\sharp\{i,j\}_\rho$ denote the number of bar codes of type 
$\{i, j\}$ for a representation $\rho.$
We have the following proposition which can be derived using the inspection  of the transformations described above.

\begin{proposition}\label {P7} $($\cite{BD11}$)$
~
\begin{enumerate}
\item 
$\sharp (i,i+1)_\rho= \dim \ker\beta_i\cap\ker\alpha_{i+1}$
\item
 $\sharp [i,i]_\rho= \dim (V_{2i} / ((\beta_i (V_{2i+1}) + \alpha_i (V_{2i-1}))$
\item 
$\sharp (i,i+1]_\rho =
\dim(\beta_i(V_{2i+1}) +\alpha_i(\ker\beta_{i-1}))   -  \dim(\beta_i(V_{2i+1}))$
\item 
$\sharp [i, i+1)_\rho =
\dim(\alpha_i(V_{2i-1}) +\beta_i(\ker\alpha_{i+1}))-\dim(\alpha_i(V_{2i-1}))$
\end{enumerate}
\end{proposition}

Note that  unless at least one elimination is performed each of these transformation  is ineffective (i.e.$= Id.$) so an algorithm based on successive applications of the transformations eventually stops.

\section{Proof of the main results but Theorem~\ref{P1}}\label{S5}

Since a tame real valued map can be regarded as a tame angle valued map by identifying $\mathbb R$ to  an open subset of 
$S^1,$ we will consider only  tame angle valued maps.

Let $f\colon X\to S^1$ be a tame map with $m$ critical angles $s_1,s_2,\dotsc,s_m$ and regular angles $t_1,t_2,\dotsc,t_m$.
First observe that, up to homotopy, the space $X$ and the map $f\colon X\to S^1$ can be 
regarded as the iterated mapping torus $\mathcal T$ and the map  $f^{\mathcal T}\to[0,m]/{\sim}$ described below. 
Consider the collection of spaces and continuous maps:
\begin{equation}\label{C}
X_m=X_0\xleftarrow{b_0=b_m}R_1
\xrightarrow{a_1}X_1
\xleftarrow{b_1}R_2
\xrightarrow{a_2}X_2\leftarrow\cdots\rightarrow X_{m-1}
\xleftarrow{b_{m-1}}R_m 
\xrightarrow{a_m}X_m
\end{equation}
with $R_i:=X_{t_i}$ and $X_i:=X_{s_i}$ and denote by $\mathcal T=T(\alpha_1 \cdots \alpha_m;\beta_1\cdots\beta_m)$ the space obtained from the disjoint union 
$$
\Bigl(\bigsqcup_{1\leq i\leq m}R_i\times[0,1]\Bigr)\sqcup\Bigl(\bigsqcup_{1\leq i\leq m}X_i\Bigr)
$$
by identifying $R_i\times\{1\}$ to $X_i$ by $\alpha_i$ and $R_i\times\{0\}$ to $X_{i-1}$ by $\beta_{i-1}$. 
Denote by $f^{\mathcal T}\colon\mathcal T\to[0,m]/{\sim}=S^1$ where $f^{\mathcal T}\colon R_i\times[0,1]\to[i-1,i]$ is the projection on
$[0,1]$ followed by the translation of $[0,1]$ to $[i-1,i]$ and $ [0,m]/{\sim}$ the space obtained from the segment $[0,m]$ by identifying the ends. 
This map is a {\it homotopical reconstruction} of $f\colon X\to S^1$ provided that, with the choice of angles  
$t_i$, $s_i,$ the maps $a_i$, $b_i$ are those described in section~\ref{S1} for $X_i:=f^{-1}(s_i)$ and  $R_i:=f^{-1}(t_i)$.

Let $\mathcal P'$ denote the space obtained from the disjoint union 
$$
\Bigl(\bigsqcup_{1\leq i\leq m}R_i\times(\epsilon,1]\Bigr)\sqcup\Bigl(\bigsqcup_{1\leq i\leq m} X_i\Bigr)
$$
by identifying $R_i\times\{1\}$ to $X_i$ by $\alpha_i$, and $\mathcal P''$ denote the space obtained from the disjoint union 
$$
\Bigl(\bigsqcup_{1\leq i\leq m}R_i\times[0,1-\epsilon\Bigr)\sqcup\Bigl(\bigsqcup_{1\leq i\leq m}X_i\Bigr)
$$
by identifying $R_i\times\{0\}$ to $X_{i-1}$ by $\beta_{i-1}$.

Let $\mathcal R=\bigsqcup_{1\leq i\leq m}R_i$ and $\mathcal X=\bigsqcup_{1\leq i\leq m}X_i$. Then, one has:
\begin{enumerate}[(i)]
\item 
$\mathcal T=\mathcal P'\cup\mathcal P''$,
\item 
$\mathcal P'\cap\mathcal P''=\bigl(\bigsqcup_{1\leq i\leq m}R_i\times(\epsilon,1-\epsilon)\bigr)\sqcup\mathcal X$, and
\item 
the inclusions $\bigl(\bigsqcup_{1\leq i\leq m}R_i\times\{1/2\}\bigr)\sqcup\mathcal X\subset\mathcal P'\cap\mathcal P''$
as well as the obvious inclusions $\mathcal X\subset\mathcal P'$ and $\mathcal X\subset\mathcal P''$ are homotopy equivalences.
\end{enumerate}
The Mayer--Vietoris long exact sequence applied to $\mathcal T= \mathcal P'\cup \mathcal P''$ leads to the diagram:
\begin{center}
$$
\xymatrix{
&                                                       & H_r(\mathcal R)\ar[r]^{M_r(\rho_r)}                                                  & H_r(\mathcal X)\ar[rd]                                               &&\\
\cdots\ar[r]&H_{r+1}(\mathcal T)\ar[ur]
\ar[r]^-{\partial_{r+1}}&H_r(\mathcal R)\oplus H_r(\mathcal X)\ar[u]^{pr_1}\ar[r]^N  &H_r(\mathcal X)\oplus H_r(\mathcal X)\ar[u]^{(Id,-Id)}{\ar[r]^-{(i^r, -i^r)}} &H_r(\mathcal T)\ar[r] &\\
&                                                       &H_r(\mathcal X)\ar[u]^{in_2}\ar[r]^{Id}                                           &H_r(\mathcal X)\ar[u]^{\Delta}                                                 &&
}
.$$
Diagram~2
\end{center}

Here $\Delta$ denotes the diagonal, $in_2$ the inclusion on the second component, $pr_1$ the projection on the first component, 
$i^r$ the linear map induced in homology by the inclusion $\mathcal X\subset \mathcal T.$ The matrix  $M_r(\alpha,\beta)$ is defined by 
\begin{equation*}
\begin{pmatrix}
\alpha^r_1 & -\beta^r_1 & 0          & \cdots         & 0\\
0          & \alpha^r_2 & -\beta^r_2 & \ddots         & \vdots\\
\vdots     & \ddots     & \ddots     & \ddots         & 0 \\
0          & \cdots     & 0          & \alpha^r_{m-1} & -\beta^r_{m-1}\\
-\beta^r_m & 0          & \cdots     & 0              & \alpha^r_m
\end{pmatrix}_. 
\end{equation*}
with $\alpha^r_i\colon H_r(R_i)\to H_r(X_i)$ and $\beta^r_i\colon H_r(R_{i+1})\to H_r(X_i)$ induced by the maps $\alpha_i$ and $\beta_i$ and the matrix $N$ is defined by 
\begin{equation*}
\begin{pmatrix}
\alpha^r & \Id\\
-\beta^r & \Id
\end{pmatrix}
\end{equation*}
where $\alpha^r$ and $\beta^r$ are the matrices 
$$
\begin{pmatrix}
\alpha^r_1 & 0          & \cdots &    0 \\
0          & \alpha^r_2 & \ddots &\vdots \\
\vdots     & \ddots     & \ddots &0 \\
0          & \cdots     & 0      & \alpha^r_{m-1} 
\end{pmatrix}
\quad\text{and}\quad
\begin{pmatrix}
0         & \beta^r_1 & 0         & \dots  & 0 \\
0         & 0         & \beta^r_2 & \ddots & \vdots \\
\vdots    & \vdots    & \ddots    & \ddots & 0 \\
0         & 0         & \dots     & 0      & \beta^r_{m-1}\\
\beta^r_m & 0         & \dots     & 0      & 0
\end{pmatrix}_.
$$
As a consequence the long exact sequence \footnote{In subsequent papers this long exact sequence is referred to as the {\it canonical sequence} associated with a tame  real or circle valued map. We like to regard it as an analogue of the Morse complex associated to a generic  gradient like vector field for a Morse real or circle valued map.} 
\begin{equation}\label{MV}
\cdots\to H_r(\mathcal R)\xrightarrow{M(\rho_r)}H_r(\mathcal X)\to H_r(\mathcal T)\to H_{r-1}(\mathcal R)\xrightarrow{M(\rho_{r-1})}H_{r-1}(\mathcal X)\to\cdots
\end{equation}
from Diagram~2 implies the short exact sequence 
\begin{equation}\label{EE13}
0\to\coker M(\rho_r)\to H_r(\mathcal T)\to\ker M(\rho_{r-1})\to0
\end{equation}
and then the  noncanonical isomorphism
\begin{equation} \label{E11}
H_r(\mathcal T)=\coker M(\rho_r)\oplus\ker M(\rho_{r-1}).
\end{equation}
Any splitting $s\colon\ker M(\rho_{r-1})\to H_r(\mathcal T)$ in the short exact sequence~\eqref{EE13} provides an isomorphism~\eqref{E11}.
Note that the long exact sequence \eqref{MV} holds also for homology with local coefficients (i.e. homology with coefficients in a representation). Such sequence   can be derived from a similar diagram as  Diagram 2, where instead of homology with coefficients in $\kappa$ one uses homology with local coefficients;  in case of interest to us with coefficients in 
$u\xi_f.$  In order to calculate 
$H_r(X;u\xi_f)$, $u\in\kappa\setminus 0$, we will use this new diagram. Since the local coefficients system $u\xi_f,$ when restricted to 
$X_s$ for any $s\in S^1$ is  trivial,  in this new diagram  all  vector spaces and linear maps but $M(\rho_r)$ remain the same as in Diagram 2. The  map  
 $M(\rho_r)$ gets  replaced by  $M((\rho_r)_u)$.  In the matrix $M((\rho_r)_u)$ all $\beta_i$ and all $\alpha_i$ but $\alpha_1$ are the same as in $M(\rho_r)$ with $\alpha_1$
replaced by the composition 
\[ 
H_r(X_{1})\xrightarrow{\alpha_1}H_r(X_{2})\xrightarrow{u}H_r(X_{2}).
\] 
The second arrow is induced by the multiplication by $u$ on the field $\kappa$. As above one obtains the non canonical isomorphism
\begin{equation}\label{E12}
H_r(X;u\xi_f)=\coker M((\rho_r)_u)\oplus\ker M((\rho_{(r-1)})_u)
\end{equation}

Theorem~\ref{T1}: Parts 1 a. and 2a. are  a straightforward consequence of Propositions~\ref{PR}(i), \ref{AP01}, \ref{AP1} and the regularity of the Jordan blocks 
representations. Parts  1 b. and 2.b are  a consequence of equation \eqref{E11} and of Propositions~\ref{O37} and \ref{O38}. Parts 1 c. and 2 c. are a particular case of Theorem~\ref{T3} 
parts 1 b.and 1.c. 

Theorem~\ref{T2}: Part (1) is a consequence of equation \eqref{E12} and of Proposition~\ref{O38}.
Part (2) follows from Theorem~ \ref{T3} parts 2. and  3. 
Note that Theorem~\ref{T1}( 1 b.) is also a consequence of Theorem~\ref{T2} Part 1) for $u=1$.

A few additional observations are necessary for the proof of Theorem~\ref{T3}.

A collection of topological spaces and continuous maps $\{X_i, R_i, a_i\colon R_i\to X_i, $ 
\newline $b_i\colon R_{i+1}\to X_i, i\in\mathbb Z\}$  
with  $X_i=\emptyset$, $i\leq n-1$, $i\geq m+1$ and $R_i=\emptyset$, $i\leq n$, $i\geq m+1$ 
 can be regarded as a collection  (\ref{C})  considered at the beginning of the section.
This is exactly what we obtain from a tame real valued map whose critical points are indexed by the integers between $n$ and $m,$   in particular 
for $\tilde f\colon\tX_{[n,m]}\to\mathbb R,$ after composing with a homeomorphism of $\mathbb R$ to make  the critical values of $\tilde f$ indexed by integers. 
The naturality of the sequence \eqref{MV} leads, for $c$, $a$, $b$, $d$ critical values with 
$c\leq a\leq b\leq d$, to the following commutative diagram
\begin{equation}
\xymatrix{
0\ar[r] & \coker M(T_{a,b}(\tilde \rho_{r}))\ar[d]^{v_l} \ar[r] &H_r(\tilde X_{[a,b]}) \ar[r]\ar[d]^v &\ker M(T_{a,b}(\tilde \rho_{r-1}))\ar[r]\ar[d]^{v_r}& 0\\
0\ar[r] & \coker M(T_{c,d}(\tilde \rho_{r}))\ar[r] &H_r(\tilde X_{[c,d]}) \ar[r] &\ker M(T_{c,d}(\tilde \rho_{r-1}))\ar[r]& 0}
\end{equation}
with $v$ induced by inclusion and $v_r$ injective. 

Indeed,  given a decomposition of the $G_{2m}-$ representation $\rho_r$ as a sum of barcodes and Jordan cells,  for any $[a,b],$ the $\mathcal Z-$ representation with compact support $T_{a,b}(\tilde \rho_{r-1}))$ has a decomposition as a sum of bar codes. The open bar codes in this decomposition, in view of  Observation~\ref{O39} are exactly   $$\{I =(\alpha, \beta)\in \tilde{\mathcal B}^o_{r-1} \mid I \subset [a,b]\}.$$ In view of 
of Proposition~\ref{O37}  one obtains a base  
in $\ker M(T_{a,b}(\tilde \rho_{r-1}))$ indexed by these open bar codes, say $e^{a,b}_I.$  
Note that $v_r$ sends $e^{a,b}_I$ into $e^{c,d}_I$ for any $I$ with $I=(\alpha,\beta) \subset [a,b].$   This shows the injectivity of $v_r.$  

If $I=(\alpha,\beta),$ choose $s_I\in H_r(\tilde X_{[\alpha,\beta]})$  to be a lift of of $e_I^{\alpha, \beta}\in \ker (M(T_{[\alpha, \beta]})$  w.r. to the surjective map $H_r(\tilde X_{[\alpha, \beta]}\to \ker (M(T_{[\alpha, \beta]}).$  For each $[a,b]$define the splitting  
$$s_{[a,b]}\colon\ker(M(T_{a,b}(\tilde \rho_{r-1})))\to H_r(\tilde X_{[a,b]})$$ by assigning to $e^{a,b}_I$ $I\in \tilde{\mathcal B}^o_{r-1} \mid I \subset (a,b)$  the image of $s_I$
in $H_r(\tilde X_{[a,b]}$ by the linear map induced by the inclusion $[\alpha, \beta]\subseteq [a,b].$

This shows that it is possible to choose  splittings 
$$
s_{[a,b]}\colon\ker(M(T_{a,b}(\tilde \rho_{r-1})))\to H_r(\tilde X_{[a,b]})
$$
and 
$$
s_{[c,d]}\colon\ker(M(T_{c,d}(\tilde \rho_{r-1})))\to H_r(\tilde X_{[c,d]}),
$$ 
satisfying $v\cdot s_{[a,b]}= s_{[c,d]}\cdot v_l$, and this for all pairs of critical values.

Note that:
\begin{enumerate}[(i)]
\item 
In view of tameness of $f$ it suffices to prove Theorem~\ref{T3} only for $a$, $b$ critical values of $\tilde f$, i.e.\ $\pi(a)$, $\pi(b)$ critical angles. 
\item 
$H_r(\tilde X)=\lim _{n\to\infty}H_r(\tilde X_{[a(n), b(n)]})$ with $a(n)$, $b(n)$ critical values of $\tilde f$ and $\lim_{n\to\infty}a(n)=-\infty$,
$\lim_{n\to\infty}b(n)=\infty$.
\item 
$\rho_r(\tilde f|_{X_{[a,b]}})=T_{a,b}(\tilde \rho_r(f))$ and Observation~\ref{O39} calculates the closed and open bar codes of $T_{a,b}(\tilde \rho_r(f))$.
\end{enumerate}

\begin{obs}\label{O51} 
1. Choose a decomposition of $\rho_r$ and $\rho_{r-1}$ in indecomposable  components and a splitting $s\colon\ker(M(\rho_{r-1})_u)\to H_r(\mathcal T;u\xi)$,
$\xi=\xi_{f^{\mathcal T}}$ in the short exact sequence 
$$
0\to\coker M((\rho_r)_u)\to H_r(\mathcal T)\to\ker M((\rho_{r-1})_u)\to0,
$$
resp.  compatible splittings $s_{[a,b]}\colon\ker (M(T_{a,b}(\tilde\rho_{r-1}))\to H_r(\tilde X_{[a,b]})$ in the short exact sequences 
$$
0\to\coker M(T_{a,b}(\tilde \rho_{r}))\to H_r(\tilde X_{[a,b]})\to\ker M(T_{a,b}(\tilde \rho_{r-1}))\to0.
$$
In view of Proposition~\ref{O38} and Observation~\ref{O39} 
one obtains the canonical isomorphisms 

\begin{equation}
\begin{aligned}
\Psi_r\colon\kappa [\{ I\in \tilde{\mathcal B}_r\mid I\ni t\}\sqcup \tilde{\mathcal J}_r] \to H_r(\tX_{t})\\
\end{aligned}
\end{equation}
for any $t\in \mathbb R$
and

\begin{equation}\label{E6}
\Psi_r\colon\kappa\bigl[\mathcal B^c_r \sqcup \mathcal B^o_{r-1} \sqcup \mathcal J^{u^{-1}}_r \sqcup \mathcal J^{u^{-1}}_{r-1}\bigr]\to H_r(\mathcal T; u\xi)
\end{equation}
resp.\
\begin{equation}\label{E66}
\Psi_r([a,b])\colon\kappa\bigl[{\mathcal B}^c (T_{a,b }(\tilde \rho_r))\sqcup {\mathcal B}^o(T_{a,b}(\tilde\rho_{r-1})\bigr]\to H_r(\tilde X_{[a,b]}).
\end{equation}
for any two critical values $a, b.$

2. Suppose $c\leq a\leq b\leq d$ are critical values. The following diagram is commutative.
\begin{center}
$$
\begin{CD} 
\kappa\bigl[\mathcal B^c(T_{a,b}(\tilde \rho_r))\sqcup \mathcal B^o(T_{a,b}(\rho_{r-1})\bigr]@>\Psi_r( [a,b])>> H_r(\tilde X_{[a,b]})\\
@V \varphi VV    @VVV\\
\kappa\bigl[\mathcal B^c(T_{c,d}(\tilde \rho_r)\sqcup \mathcal B^o(T_{c,d}(\tilde\rho_{r-1})\bigr]@>\Psi_r([c,d])>> H_r(\tilde X_{[c,d]})
\end{CD}
$$
Diagram 3
\end{center}
The right side vertical arrow in Diagram 3  is induced by inclusion and the left side vertical arrow 
$$
\varphi\colon\kappa[\mathcal B^c(T_{a,b}(\tilde\rho_r))] \oplus \kappa[ \mathcal B^o(T_{a,b}(\tilde\rho_{r-1}))]  \to \kappa[\mathcal B^c(T_{c,d}(\tilde\rho_r))] \oplus k[ \mathcal B^o(T_{c,d}(\rho_{r-1}))]
$$ 
is the direct sum of the linear maps 
$$
\varphi_1\colon\kappa[\mathcal B^c(T_{a,b}(\tilde\rho_r))] \to \kappa[\mathcal B^c(T_{c,d}(\tilde\rho_r))],
\quad
\varphi_2\colon\kappa[\mathcal B^o(T_{a,b}(\tilde\rho_r))] \to \kappa[\mathcal B^o(T_{c,d}(\tilde\rho_r))]
.$$
The map $\varphi_2$ is induced by inclusion and $\varphi_1$ is the linear extension of the map defined on $\mathcal B^c(T_{a,b}(\tilde\rho_r))$ as follows.
If an $I\in\mathcal B^c(T_{a,b}(\tilde\rho_r))$ remains an element in $\mathcal B^c(T_{c,d}(\tilde\rho_r))$ then $\varphi_1 (I)=I$, if not $\varphi_1 (I)=0.$ 
\end{obs}

Observations~\ref{O51} and item (ii) above lead to the commutative diagram
\begin{center}
$$
\begin{CD} 
\kappa\bigl[\mathcal B^c(T_{a,b}(\tilde \rho_r))\sqcup \mathcal B^o(T_{a,b}(\tilde \rho_{r-1})\bigr]@>\Psi_r( [a,b])>> H_r(\tilde X_{[a,b]})\\
@V \varphi' VV    @VVV\\
\kappa\bigl[(\tilde{ B^c_r}\sqcup\tilde{\mathcal J}_r) \sqcup \tilde{\mathcal B}^o_{r-1}\bigr]@>\Psi_r>> H_r( \tilde X)\\
@V \varphi'' VV    @VVV\\
\kappa\bigl[{ (B^c_r}\sqcup{\mathcal J}^1_r)\sqcup ({\mathcal B}^o_{r-1}\sqcup \mathcal J^1_{r-1})\bigr]@>\hat\Psi_r>> H_r( X).
\end{CD}
$$
Diagram~4
\end{center}
The right side vertical arrows are induced by inclusion and by the covering map $p\colon\tilde X\to X$, and the left side vertical arrows are defined as follows.

The map 
$$
\varphi'\colon\kappa[\mathcal B^c(T_{a,b}(\tilde \rho_r))\sqcup \mathcal B^o(T_{a,b}(\rho_{r-1})) ] \to \kappa[ (\tilde{ B^c_r}\sqcup\tilde{\mathcal J}_r) \sqcup \tilde{\mathcal B}^o_{r-1}] 
$$  
is the direct sum of the linear maps 
$$
\varphi'_1\colon\kappa[\mathcal B^c(T_{a,b}(\tilde \rho_r))] \to \kappa[ \tilde{ B^c_r}\sqcup\tilde{\mathcal J}_r],
\qquad
\varphi'_2\colon\kappa[ \mathcal B^o(T_{a,b}(\rho_{r-1}))]\to  \kappa[ (B^o_{r-1})],
$$
and the map 
$$
\varphi''\colon \kappa[ (\tilde{ B^c_r}\sqcup\tilde{\mathcal J}_r) \sqcup \tilde{\mathcal B}^o_{r-1}] 
\to k[{ (B^c_r}\sqcup{\mathcal J}^1_r)\sqcup ({\mathcal B}^o_{r-1}\sqcup \mathcal J^1_{r-1})] 
$$ 
is the direct sum of the linear maps 
$$
\varphi''_1\colon \kappa[ (\tilde{ B^c_r}\sqcup\tilde{\mathcal J}_r) ] \to \kappa[\mathcal B^c_r \sqcup \mathcal J^1_r)],
\qquad
\varphi''_2\colon\kappa[ \tilde{\mathcal B}^o_{r-1}]\to \kappa[ \mathcal B^o_{r-1}].
$$ 
The map $\varphi'_2$ is induced by inclusion and $\varphi'_1$ is the linear extension on the map defined on
$\tilde{ B^c_r}\sqcup\tilde{\mathcal J}_r$ as follows. If $I$ is an element of $\tilde{B^c_r}\sqcup\tilde{\mathcal J}_r$ which is actually an element of  
$\tilde B^c_r$ or an element of $\tilde {\mathcal J}_r$ then $\varphi'_1(I)=I$, otherwise $\varphi'_1(I)=0$. 
The maps $\varphi''_1$ and $\varphi''_2$ are the linear extensions of the maps defined on $(\tilde{B^c_r}\sqcup\tilde{\mathcal J}_r)$ 
and on $\tilde{\mathcal B}^o_{r-1}$ as follows.

An element $I+2\pi k \in \tilde{ \mathcal B}^c_r$ with $I\in\mathcal B^c_r$ is sent by $\varphi_1''$ to $I$ and an element in  
$\tilde{\mathcal J}_r$ which corresponds to the Jordan cell in $J\in \mathcal J^1_r$ is sent to $J$. All other elements are sent to zero.

An element $I+2\pi k\in\tilde{\mathcal B}^o_{r-1}$ with $I\in\mathcal B^o_{r-1}$ is sent by $\varphi_2''$ to $I$.


Theorem~\ref{T3} (Part 1) follows from Diagram~4 by inspecting its left side. To derive (Part 2) and (Part 3) observe that the additive group of integers  
$\mathbb Z$ acts on the set $\tilde{\mathcal B}^c_r \sqcup \tilde{\mathcal B}^o_r$ freely by translation  with the quotient set  
$\mathcal B^c_r \sqcup \mathcal B^o_r$  and trivially on $\tilde {\mathcal J}_r$. The $\mathbb Z[T^{-1},T]$-module structure of $H_r(\tX)$  
corresponds via $\Psi$ to the module structure on $\kappa[(\tilde{\mathcal B} ^c_r \sqcup (\tilde {\mathcal J}_r)\sqcup (\tilde{\mathcal B}^o_{r-1})]$ 
induced by these actions. Theorem~\ref{T4} (Part 2) follows from Theorem \ref{T3} and  (Part 1) from Theorem~\ref{T4} (Part 2) and Theorem~\ref{T1}.

\section{Proof of Theorem~\ref{P1}}\label{S6}

Suppose $f\colon X\to S^1$ is a continuous map. Let $\theta\in S^1$ be a tame value \footnote{i.e. $X_\theta$ is a deformation retract of an open neighborhood of $X_\theta$} 
and denote its level by $X_\theta=f^{-1}(\theta)$. Moreover, let $H_*(X_\theta)$
denote its singular homology with coefficients in any fixed unital ring $\kappa$ which is a $\kappa-$ module (vector space when $\kappa$ is a field).
To this situation we will associate a linear relation,
$$
R\colon H_*(X_\theta)\leadsto H_*(X_\theta),
$$
see section~\ref{SS:mon} below. One can think of a linear relation as a partially defined, 
multivalued linear map, see section~\ref{SS:reg} below.
While this relation $R$ depends very much on the tame value $\theta$ and the function $f$, its
regular part (a linear isomorphism to be defined in section~\ref{SS:reg}),
$$
R_\reg\colon H_*(X_\theta)_\reg\xrightarrow\cong H_*(X_\theta)_\reg,
$$
turns out to be independent on $\theta$ and a homotopy invariant of $f.$. More precisely, we will show that $R_\reg$ coincides with the monodromy induced by the deck transformation 
on a certain invariant submodule of $H_*(\tilde X)$, where $\tilde X$ denotes the infinite cyclic covering associated with $f$. 
For the precise statement see Theorem~\ref{T:monreg} below. As a corollary of these considerations we obtain a proof of Theorem~\ref{P1}.

\subsection{Linear relations and their regular part}\label{SS:reg}

Suppose $V$ and $W$ are two modules over a fixed commutative ring.
Recall that a linear relation from $V$ to $W$ can be considered as a submodule $R\subseteq V\times W$. 
Notationally, we indicate this situation by $R\colon V\leadsto W$. For $v\in V$ and $w\in W$ we write $vRw$ iff $v$ is in 
relation with $w$, i.e.\ $(v,w)\in R$. Every module homomorphism $V\to W$ can be regarded as a linear relation $V\leadsto W$ 
in a natural way. If $U$ is another module, and $S\colon W\leadsto U$ is a linear relation, then
the composition $SR\colon V\leadsto U$ is the linear relation defined by $v(SR)u$ iff there exists $w\in W$ 
such that $vRw$ and $wSu$. Clearly, this is an associative composition generalizing the ordinary composition of module
homomorphisms. For the identical relations we have $R\id_V=R$ and $\id_WR=R$. Modules over a fixed commutative ring and linear relations 
thus constitute a category. If $R\colon V\leadsto W$ is a linear relation we define a linear relation $R^\dag\colon W\leadsto V$ 
by $wR^\dag v$ iff $vRw$. Clearly, $R^{\dag\dag}=R$ and $(SR)^\dag=R^\dag S^\dag$.

A linear relation $R\colon V\leadsto W$ gives rise to the following submodules:
\begin{align*}
\dom(R)&:=\{v\in V\mid\exists w\in W:vRw\}
\\
\img(R)&:=\{w\in W\mid\exists v\in V:vRw\}
\\
\ker(R)&:=\{v\in V\mid vR0\}
\\
\mul(R)&:=\{w\in W\mid 0Rw\}
\end{align*}
Clearly, $\ker(R)\subseteq\dom(R)\subseteq V$, and $W\supseteq\img(R)\supseteq\mul(R)$.
Note that $R$ is a homomorphism (map) iff $\dom(R)=V$ and $\mul(R)=0$. One readily verifies:

\begin{lemma}\label{L:1}
For a linear relation $R\colon V\leadsto W$ the following are equivalent:
\begin{enumerate}[(a)]
\item
$R$ is an isomorphism in the category of modules and linear relations.
\item
$\dom(R)=V$, $\img(R)=W$, $\ker(R)=0$, and $\mul(R)=0$.
\item
$R$ is an isomorphism of modules.
\end{enumerate}
In this case $R^{-1}=R^\dag$.
\end{lemma}

For a linear relation $R\colon V\leadsto V$, we introduce the following submodules:
\begin{align*}
K_+&:=\{v\in V\mid\exists k\,\exists v_i\in V:vRv_1Rv_2R\cdots Rv_kR0\}
\\
K_-&:=\{v\in V\mid\exists k\,\exists v_i\in V:0Rv_{-k}R\cdots Rv_{-2}Rv_{-1}Rv\}
\\
D_+&:=\{v\in V\mid\exists v_i\in V:vRv_1Rv_2Rv_3R\cdots\}
\\
D_-&:=\{v\in V\mid\exists v_i\in V:\cdots Rv_{-3}Rv_{-2}Rv_{-1}Rv\}
\\
D:=D_-\cap D_+&=\{v\in V\mid\exists v_i\in V:\cdots Rv_{-2}Rv_{-1}RvRv_1Rv_2R\cdots\},
\end{align*}
Clearly, $K_-\subseteq D_-\subseteq V\supseteq D_+\supseteq K_+$.
Also note that passing from $R$ to $R^\dag$, the roles of $+$ and $-$ get interchanged.
Moreover, we introduce a linear relation on the quotient module
\begin{equation*} 
V_\reg:=\frac{D}{(K_-+K_+)\cap D}
\end{equation*}
defined as the composition 
$$
V_\reg=\frac{D}{(K_-+K_+)\cap D}\overset{\pi^\dag}\leadsto D\overset\iota\leadsto V\overset R\leadsto V\overset{\iota^\dag}\leadsto D\overset\pi\leadsto\frac{D}{(K_-+K_+)\cap D}=V_\reg,
$$
where $\iota$ and $\pi$ denote the canonical inclusion and projection, respectively.
In other words, two elements in $V_\reg$ are related by $R_\reg$ iff they admit representatives in $D$
which are in related by $R$. We refer to $R_\reg$ as the \emph{regular part} of $R$.

\begin{proposition}\label{P:AA}
The relation $R_\reg\colon V_\reg\leadsto V_\reg$ is an isomorphism of modules. Moreover,
the natural inclusion induces a canonical isomorphism
\begin{equation}\label{E:100}
V_\reg=\frac{D}{(K_-+K_+)\cap D}\xrightarrow\cong\frac{(K_-+D_+)\cap(D_-+K_+)}{K_-+K_+}
\end{equation}
which intertwines $R_\reg$ with the relation induced on the right hand side quotient.
\end{proposition}

\begin{proof}
Clearly, \eqref{E:100} is well defined and injective. To see that it is onto let
$$
x=k_-+d_+=d_-+k_+\in(K_-+D_+)\cap(D_-+K_+),
$$
where $k_\pm\in K_\pm$ and $d_\pm\in D_\pm$. Thus
$$
x-k_--k_+=d_+-k_+=d_--k_-\in D_-\cap D_+=D.
$$
We conclude $x\in D+K_-+K_+$, whence \eqref{E:100} is onto. We will next show that this isomorphism intertwines
$R_\reg$ with the relation induced on the right hand side. To do so, suppose $xR\tilde x$ where
\begin{align*}
x&=k_-+d_+=d_-+k_+\in(K_-+D_+)\cap(D_-+K_+),
\\
\tilde x&=\tilde k_-+\tilde d_+=\tilde d_-+\tilde k_+\in(K_-+D_+)\cap(D_-+K_+),
\end{align*}
and $k_\pm,\tilde k_\pm\in K_\pm$ and $d_\pm,\tilde d_\pm\in D_\pm$. Note that there exist $k_+'\in K_+$ and
$\tilde k_-'\in K_-$ such that $k_+Rk_+'$ and $\tilde k_-'R\tilde k_-$. By linearity of $R$ we obtain
$$
\underbrace{(x-k_+-\tilde k_-')}_{\in D_-}R\underbrace{(\tilde x-k_+'-\tilde k_-)}_{\in D_+}.
$$
We conclude $d:=x-k_+-\tilde k_-'\in D$, $\tilde d:=\tilde x-k_+'-\tilde k_-\in D$, and $dR\tilde d$.
This shows that the relations induced on the two quotients in \eqref{E:100} coincide.
We complete the proof by showing that $R_\reg$ is an isomorphism.
Clearly, $\dom(R_\reg)=V_\reg=\img(R_\reg)$. We will next show $\ker(R_\reg)=0$. To this end suppose
$dR\tilde d$, where
$$
d\in D\quad\text{and}\quad\tilde d=\tilde k_-+\tilde k_+\in(K_-+K_+)\cap D
$$
with $\tilde k_\pm\in K_\pm$. Note that $\tilde k_-=\tilde d-\tilde k_+\in K_-\cap D_+$. Thus there exists
$k_-\in K_-\cap D_+$ such that $k_-R\tilde k_-$. By linearity of $R$, we get $(d-k_-)R\tilde k_+$, whence 
$d-k_-\in K_+$ and thus $d\in K_-+K_+$.
This shows $\ker(R_\reg)=0$. Analogously, we have $\mul(R_\reg)=0$. In view of Lemma~\ref{L:1} we conclude that
$R_\reg$ is an isomorphism of modules. 
\end{proof}

We will now specialize to linear relations on finite dimensional vector spaces and provide
another description of $V_\reg$ in this case. Consider the category whose objects are finite 
dimensional vector spaces $V$ equipped with a linear relation $R\colon V\leadsto V$ and whose 
morphisms are linear maps $\psi\colon V\to W$ such that for all $x,y\in V$ with $xRy$ we also have 
$\psi(x)Q\psi(y)$, where $W$ is another finite dimensional vector space with linear relation $Q\colon W\leadsto W$.
It is readily checked that this is an abelian category. By the Remak--Schmidt theorem, every linear
relation on a finite dimensional vector space can therefore be decomposed into a direct sum of
indecomposable ones, $R\cong R_1\oplus\cdots\oplus R_N$, where the factors are unique up to
permutation and isomorphism. The decomposition itself, however, is not canonical.

\begin{proposition}\label{P:C}
Let $R\colon V\leadsto V$ be a linear relation on a finite dimensional vector space over an algebraic closed field , and let
$R\cong R_1\oplus\cdots\oplus R_N$ denote a decomposition into indecomposable linear relations.
Then $R_\reg$ is isomorphic to the direct sum of factors $R_i$ whose relations are linear isomorphisms.
\end{proposition}

\begin{proof}
Since the definition of $R_\reg$ is a natural one, we clearly have
$$
R_\reg\cong(R_1)_\reg\oplus\cdots\oplus(R_N)_\reg.
$$
Consequently, it suffices to show the following two assertions:
\begin{enumerate}[(a)]
\item
If $R\colon V\leadsto V$ is an isomorphism of vector spaces, then $V_\reg=V$ and $R_\reg=R$.
\item
If $R\colon V\leadsto V$ is an indecomposable linear relation on a finite dimensional vector space which is not a linear isomorphism, then $V_\reg=0$.
\end{enumerate}
The first statement is obvious, in this case we have $K_-=K_+=0$ and $D=D_-=D_+=V$. 
To see the second assertion, note that an indecomposable linear relation $R\subseteq V\times V$ gives rise to an indecomposable representation $R\genfrac{}{}{0pt}{}{\to}{\to}V$ of the quiver $G_2$.
Since $R$ is not an isomorphism, the quiver representation has to be of the bar code type.
Using the explicit descriptions of the bar code representations, it is straight forward to conclude $V_\reg=0$.
\end{proof}

In the subsequent section we will also make use of the following result:

\begin{proposition}\label{P:X}
Suppose $R\colon V\leadsto V$ is a linear relation on a finite dimensional vector space. Then:
\begin{equation}\label{E:12}
D_+=D+K_+,\quad D_-=K_-+D,\quad\text{and}
\end{equation}
\begin{equation}\label{E:13}
K_-\cap D_+=K_-\cap K_+=D_-\cap K_+.
\end{equation}
\end{proposition}

For the proof we first establish two lemmas.

\begin{lemma}\label{L:2}
Suppose $R\colon V\leadsto W$ is a linear relation between vector spaces such that $\dim V=\dim W<\infty$. 
Then the following are equivalent:
\begin{enumerate}[(a)]
\item
$R$ is an isomorphism.
\item
$\dom(R)=V$ and $\ker(R)=0$.
\item
$\img(R)=W$ and $\mul(R)=0$.
\end{enumerate}
\end{lemma}

\begin{proof}
This follows immediately from the dimension formula
$$
\dim\dom(R)+\dim\mul(R)=\dim(R)=\dim\img(R)+\dim\ker(R)
$$
and Lemma~\ref{L:1}.
\end{proof}

\begin{lemma}\label{L:3}
If $V$ is finite dimensional, then the composition of relations
$$
D_+/K_+\overset{\pi^\dag}\leadsto D_+\overset\iota\leadsto V\overset{R^k}\leadsto V\overset{\iota^\dag}\leadsto D_+\overset\pi\leadsto D_+/K_+,
$$
is a linear isomorphism, for every $k\geq0$, where $\iota$ and $\pi$ denote the canonical inclusion and projection, respectively.
Analogously, the relation induced by $R^k$ on $D_-/K_-$ is an isomorphism, for all $k\geq0$. Moreover, for sufficiently large $k$,
$$
D_-=\img(R^k)\quad\text{and}\quad D_+=\dom(R^k).
$$
\end{lemma}

\begin{proof}
One readily verifies $\dom(\pi\iota^\dag R^k\iota\pi^\dag)=D_+/K_+$ and $\ker(\pi\iota^\dag R^k\iota\pi^\dag)=0$.
The first assertion thus follows from Lemma~\ref{L:2} above. Considering $R^\dag$ we obtain the second statement.
Clearly, $\dom(R^k)\supseteq\dom(R^{k+1})$, for all $k\geq0$.
Since $V$ is finite dimensional, we must have $\dom(R^k)=\dom(R^{k+1})$, for sufficiently large $k$.
Given $v\in\dom(R^k)$, we thus find $v_1\in\dom(R^k)$ such that $vRv_1$. Proceeding inductively, we construct $v_i\in\img(R^k)$ such that
$vRv_1Rv_2R\cdots$, whence $v\in D_+$. This shows $\dom(R^k)\subseteq D_+$, for sufficiently large $k$. As the converse inclusion is obvious
we get $D_+=\dom(R^k)$. Considering $R^\dag$, we obtain the last statement.
\end{proof}

\begin{proof}[Proof of Proposition~\ref{P:X}]
From Lemma~\ref{L:3} we get $\img(\pi\iota^\dag R^k)=D_+/K_+$, whence
$D_+\subseteq\img(R^k)+K_+$, for every $k\geq0$, and thus $D_+\subseteq D_-+K_+$. This implies
$D_+=D+K_+$. Considering $R^\dag$ we obtain the other equality in \eqref{E:12}.
From Lemma~\ref{L:3} we also get $\mul(\pi\iota^\dag R^k)=0$, whence
$\mul(R^k)\cap D_+\subseteq K_+$, for every $k\geq0$. This gives 
$K_-\cap D_+=K_-\cap K_+$. Considering $R^\dag$ we get the other equality in \eqref{E:13}.
\end{proof}

\subsection{Monodromy}\label{SS:mon}

Suppose $f: X\to S^1$ is a continuous map and let
$$
\xymatrix{
\tilde X\ar[d]\ar[r]^-{\tilde f}&\R\ar[d]\\X\ar[r]^-f&S^1
}
$$
denote the associated infinite cyclic covering. For $r\in\R$ we put $\tilde X_r=\tilde f^{-1}(r)$ and let $H_*(\tilde X_r)$ denote
its singular homology with coefficients in any fixed module. If $r_1\leq r_2$ we define a linear relation 
$$
B_{r_1}^{r_2}\colon H_*(\tilde X
_{r_1})\leadsto H_*(\tilde X
_{r_2})
$$ 
by declaring $a_1\in H_*(\tilde X
_{r_1})$ to be in relation with $a_2\in H_*(\tilde X
_{r_2})$ iff
their images in $H_*(\tilde X_{[r_1,r_2]})$ coincide, where $\tilde X_{[r_1,r_2]}=f^{-1}([r_1,r_2])$. 
If $r_1\leq r_2\leq r_3$ we clearly have $B_{r_2}^{r_3}B_{r_1}^{r_2}\subseteq B_{r_1}^{r_3}$. If $r_2$ 
is a tame value this becomes an equality of relations:

\begin{lemma}
Suppose $r_1\leq r_2\leq r_3$ and assume $r_2$ is a tame value. 
Then, as linear relations, $B_{r_2}^{r_3}B_{r_1}^{r_2}=B_{r_1}^{r_3}$.
\end{lemma}

\begin{proof}
Since $r_2$ is a tame value, we have an exact Mayer--Vietoris sequence,
$$
H_*(\tilde X
_{r_2})\to H_*(\tilde X_{[r_1,r_2]})\oplus H_*(\tilde X_{[r_2,r_3]})\to H_*(\tilde X_{[r_1,r_3]}),
$$
which immediately implies the statement.
\end{proof}

Fix a tame value $\theta\in S^1$ of $f$ and a lift $\tilde\theta\in\R$, $e^{\mathbf i\tilde\theta}=\theta$.
Using the projection $\tilde X\to X$, we may canonically identify $\tilde X_{\tilde\theta}=X_\theta=f^{-1}(\theta)$.
Moreover, let $\tau\colon\tilde X\to\tilde X$ denote the fundamental deck transformation,  i.e.\ $\tilde f\circ\tau=\tilde f+2\pi$. 
Note that $\tau$ induces homeomorphisms between levels, $\tau\colon\tilde X_r\to\tilde X_{r+2\pi}$, and define a linear relation 
$$
R\colon H_*(X_\theta)\leadsto H_*(X_\theta)
$$ 
as the composition
\begin{equation}\label{E:RSigma}
H_*(X_\theta)=H_*(\tilde X_{\tilde\theta})\overset{B_{\tilde\theta}^{\tilde\theta+2\pi}}\leadsto H_*(\tilde X_{\tilde\theta+2\pi})\overset{\tau_*^\dag}
\leadsto H_*(\tilde X_{\tilde\theta})=H_*(X_\theta).
\end{equation}
In other words, for $a,b\in H_*(X_\theta)$ we have $aRb$ iff $aB_{\tilde\theta}^{\tilde\theta+2\pi}(\tau_*b)$,
i.e.\ iff $a$ and $\tau_*b$ coincide in $H_*(\tilde X_{[\tilde\theta,\tilde\theta+2\pi]})$. Particularly:

\begin{lemma}\label{L:4}
If $a,b\in H_*(X_\theta)$ and $aRb$, then $a=\tau_*b$ in $H_*(\tilde X)$.
\end{lemma}

We will continue to use the notation $K_\pm$, $D_\pm$, and $R_\reg$ introduced in the previous section for this relation $R$ on $H_*(X_\theta)$.
Particularly, its regular part,  
$$
R_\reg\colon H_*(X_\theta)_\reg\to H_*(X_\theta)_\reg,
$$
is a module automorphism.

\begin{lemma}\label{L:5}
We have:
\begin{align*}
K_+&=\ker\bigl(H_*(X_\theta)\to H_*(\tilde X_{[\tilde\theta,\infty)})\bigr)
\\
K_-&=\ker\bigl(H_*(X_\theta)\to H_*(\tilde X_{(-\infty,\tilde\theta]})\bigr)
\end{align*}
Both maps are induced by the canonical inclusion $X_\theta=\tilde X_{\tilde\theta}\to\tilde X$.
\end{lemma}

\begin{proof}
We will only show the first equality, the other one can be proved along the same lines. To see the inclusion
$K_+\subseteq\ker(H_*(X_\theta)\to H_*(\tilde X_{[\tilde\theta,\infty)}))$, let $a\in K_+$. Hence, there exist $a_k\in H_*(X_\theta)$,
almost all of which vanish, such that $aRa_1Ra_2R\cdots$. In $H_*(\tilde X_{[\tilde\theta,\tilde\theta+2\pi]})$, we thus have:
$$
a=\tau_*a_1,\quad a_1=\tau_*a_2,\quad a_2=\tau_*a_3,\quad\dotsc 
$$
In $H_*(\tilde X_{[\tilde\theta,\infty)})$, we obtain:
$$
a=\tau_*a_1=\tau_*^2a_2=\tau_*^3a_3=\cdots
$$
Since some $a_k$ have to be zero, we conclude that $a$ vanishes in $H_*(\tilde X_{[\tilde\theta,\infty)})$.

To see the converse inclusion, $K_+\supseteq\ker(H_*(\tilde X_\theta)\to H_*(\tilde X_{[\tilde\theta,\infty)}))$, set 
$$
U:=\bigsqcup_{\text{$0\leq k$ even}}\tilde X_{[\tilde\theta+2\pi k,\tilde\theta+2\pi(k+1)]},\qquad
V:=\bigsqcup_{\text{$1\leq k$ odd}}\tilde X_{[\tilde\theta+2\pi k,\tilde\theta+2\pi(k+1)]}
$$
and note that $U\cup V=\tilde X_{[\tilde\theta,\infty)}$, as well as $U\cap V=\bigsqcup_{k\in\mathbb N}\tilde X_{\tilde\theta+2\pi k}$. 
Since $\theta$ is a tame value, we have an exact Mayer--Vietoris sequence
$$
\bigoplus_{k\in\mathbb N}H_*(\tilde X_{\tilde\theta+2\pi k})=H_*\Bigl(\bigsqcup_{k\in\mathbb N}\tilde X_{\tilde\theta+2\pi k}\Bigr)\to H_*(U)\oplus H_*(V)\to H_*(\tilde X_{[\tilde\theta,\infty)}).
$$
For $b\in\ker(H_*(X_\theta)\to H_*(\tilde X_{[\tilde\theta,\infty)}))$ we thus find $b_k\in H_*(\tilde X_{\tilde\theta+2\pi k})$, almost all of which vanish, such that:
$$
b=b_1\in H_*(\tilde X_{[\tilde\theta,\tilde\theta+2\pi]}),\quad
b_1+b_2=0\in H_*(\tilde X_{[\tilde\theta+2\pi,\tilde\theta+4\pi]}),\quad
b_2+b_3=0\in H_*(\tilde X_{[\tilde\theta+4\pi,\tilde\theta+6\pi]}),\quad\dotsc
$$
Putting $c_k:=(-1)^{k-1}\tau^{-k}_*b_k\in H_*(\tilde X_{\tilde\theta})$, we obtain the following equalities in $H_*(\tilde X_{[\tilde\theta,\tilde\theta+2\pi]})$:
$$
b=\tau_*c_1,\quad c_1=\tau_*c_2,\quad c_2=\tau_*c_3,\quad\dotsc
$$
In other words, we have the relations $bRc_1Rc_2Rc_3R\cdots$. Since some $c_k$ has to be zero, we conclude $b\in K_+$, whence the lemma.
\end{proof}

Introduce the upwards Novikov complex as a projective limit of relative singular chain complexes,
$$
C_*^{\Nov,+}(\tilde X):=\varprojlim_{r}C_*(\tilde X,\tilde X_{[r,\infty)}),
$$
and let $H_*^{\Nov,+}(\tilde X)$ denote its homology. Analogously, we define a downwards Novikov complex
$C_*^{\Nov,-}(\tilde X)=\varprojlim_r C_*(\tilde X,\tilde X_{(-\infty,r]})$ and the corresponding homology,
$H_*^{\Nov,-}(\tilde X)$. We will also use similar notation for subsets of $\tilde X$.

\begin{lemma}\label{L:6}
We have:
\begin{align*}
D_+&=\ker\bigl(H_*(X_\theta)\to H_*^{\Nov,+}(\tilde X_{[\tilde\theta,\infty)})\bigr)
\\
D_-&=\ker\bigl(H_*(X_\theta)\to H_*^{\Nov,-}(\tilde X_{(-\infty,\tilde\theta]})\bigr)
\end{align*}
Both maps are induced by the canonical inclusion $X_\theta=\tilde X_{\tilde\theta}\to\tilde X$.
\end{lemma}

\begin{proof}
Using the exact Mayer--Vietoris sequence
$$
\prod_{k\in\mathbb N}H_*(\tilde X_{\tilde\theta+2\pi k})=H_*^{\Nov,+}\Bigl(\bigsqcup_{k\in\mathbb N}\tilde X_{\tilde\theta+2\pi k}\Bigr)\to H_*^{\Nov,+}(U)\oplus H_*^{\Nov,+}(V)\to H_*^{\Nov,+}(\tilde X_{[\tilde\theta,\infty)}),
$$
this can be proved along the same lines as Lemma~\ref{L:5}.
\end{proof}

Let us introduce a complex
$$
C_*^\lf(\tilde X):=\varprojlim_r C_*(\tilde X,\tilde X_{(-\infty,-r]}\cup\tilde X_{[r,\infty)})
$$
and denote its homology by $H_*^\lf(\tilde X)$. If $f$ is proper, this is the complex of locally finite singular chains.

\begin{lemma}\label{L:7}
We have:
\begin{align*}
K_-+K_+&=\ker\bigl(H_*(X_\theta)\to H_*(\tilde X)\bigr)
\\
K_-+D_+&=\ker\bigl(H_*(X_\theta)\to H_*^{\Nov,+}(\tilde X)\bigr)
\\
D_-+K_+&=\ker\bigl(H_*(X_\theta)\to H_*^{\Nov,-}(\tilde X)\bigr)
\\
D_-+D_+&=\ker\bigl(H_*(X_\theta)\to H_*^\lf(\tilde X)\bigr)
\end{align*}
All maps are induced by the canonical inclusion $X_\theta=\tilde X_{\tilde\theta}\to\tilde X$.
\end{lemma}

\begin{proof}
The first statement follows from the exact Mayer--Vietoris sequence
$$
H_*(\tilde X_{\tilde\theta})\to H_*(\tilde X_{(-\infty,\tilde\theta]})\oplus H_*(\tilde X_{[\tilde\theta,\infty)})\to H_*(\tilde X)
$$
and Lemma~\ref{L:5}. The second assertion follows from the exact Mayer--Vietoris sequence
$$
H_*(\tilde X_{\tilde\theta})\to H_*(\tilde X_{(-\infty,\tilde\theta]})\oplus H_*^{\Nov,+}(\tilde X_{[\tilde\theta,\infty)})\to H_*^{\Nov,+}(\tilde X)
$$
and Lemma~\ref{L:5} and \ref{L:6}. Similarly, one can check the third equality. To see the last statement
we use the exact Mayer--Vietoris sequence
$$
H_*(\tilde X_{\tilde\theta})\to H_*^{\Nov,-}(\tilde X_{(-\infty,\tilde\theta]})\oplus H_*^{\Nov,+}(\tilde X_{[\tilde\theta,\infty)})\to H_*^\lf(\tilde X)
$$
and Lemma~\ref{L:6}.
\end{proof}

\begin{lemma}\label{L:8}
We have
$$
\ker\Bigl(H_*(\tilde X)\to H_*^{\Nov,-}(\tilde X)\oplus H_*^{\Nov,+}(\tilde X)\Bigr)
\subseteq\img\bigl(H_*(\tilde X_{\tilde\theta})\to H_*(\tilde X)\bigr),
$$
where all maps are induced by the tautological inclusions.
\end{lemma}

\begin{proof}
This follows from the following commutative diagram of exact Mayer--Vi\-e\-to\-ris sequences:
$$
\xymatrix{
H^\lf_{*+1}(\tilde X)\ar[r]^-\partial & H_*(\tilde X)\ar[r] & H_*^{\Nov,-}(\tilde X)\oplus H_*^{\Nov,+}(\tilde X)
\\
H^\lf_{*+1}(\tilde X)\ar@{=}[u]\ar[r]^-\partial & H_*(\tilde X_{\tilde\theta})\ar[u]\ar[r] & H_*^{\Nov,-}(\tilde X_{(-\infty,\tilde\theta]})\oplus H_*^{\Nov,+}(\tilde X_{[\tilde\theta,\infty)}) \ar[u]
}
$$
A similar argument was used in \cite[Lemma~2.5]{HL99b}.
\end{proof}

\begin{theorem}\label{T:monreg}
The inclusion $\iota\colon X_\theta=\tilde X_{\tilde\theta}\to\tilde X$ induces a canonical isomorphism
$$
H_*(X_\theta)_\reg=\frac D{(K_-+K_+)\cap D}
\xrightarrow\cong\ker\Bigl(H_*(\tilde X)\to H_*^{\Nov,-}(\tilde X)\oplus H_*^{\Nov,+}(\tilde X)\Bigr),
$$
intertwining $R_\reg$ with the monodromy isomorphism induced by the deck transformation $\tau\colon\tilde X\to\tilde X$ on the right hand side. 
Moreover, working with coefficients in a field, and assuming that $H_*(X_\theta)$ is finite dimensional, the common kernel on the right hand side above
coincides with
$$
\ker\bigl(H_*(\tilde X)\to H_*^{\Nov,-}(\tilde X)\bigr)
=\ker\bigl(H_*(\tilde X)\to H_*^{\Nov,+}(\tilde X)\bigr).
$$
Particularly, in this case the latter two kernels are finite dimensional too.
\end{theorem}

\begin{proof}
It follows immediately from Lemma~\ref{L:7} and \ref{L:8} that $\iota_*\colon H_*(X_\theta)\to H_*(\tilde X)$ induces an isomorphism
$$
\frac{(K_-+D_+)\cap(D_-+K_+)}{K_-+K_+}\xrightarrow\cong\ker\Bigl(H_*(\tilde X)\to H_*^{\Nov,-}(\tilde X)\oplus H_*^{\Nov,+}(\tilde X)\Bigr).
$$
In view of Lemma~\ref{L:4}, this isomorphism intertwines the isomorphism induced by $R$ on the left hand side,
with the monodromy isomorphism on the right hand side. Combining this with Proposition~\ref{P:AA} we obtain the first assertion. 
For the second statement it suffices to show
\begin{equation}\label{E:543}
\ker\bigl(H_*(\tilde X)\to H_*^{\Nov,+}(\tilde X)\bigr)
\subseteq\ker\Bigl(H_*(\tilde X)\to H_*^{\Nov,-}(\tilde X)\oplus H_*^{\Nov,+}(\tilde X)\Bigr),
\end{equation}
as the converse inclusion is obvious, and the corresponding statement for the 
downward Novikov homology can be derived analogously.
To this end, suppose $a\in\ker\bigl(H_*(\tilde X)\to H_*^{\Nov,+}(\tilde X)\bigr)$.
Then there exists $k$ such that $\tau^k_*a$ is contained in the image of $H_*(\tilde X_{(-\infty,\tilde\theta]})\to H_*(\tilde X)$.
Using the exact Mayer--Vietoris sequence
$$
H_*(\tilde X_{\tilde\theta})\to H_*(\tilde X_{(-\infty,\tilde\theta]})\oplus H_*^{\Nov,+}(\tilde X_{[\tilde\theta,\infty)})\to H_*^{\Nov,+}(\tilde X)
$$
we conclude, that $\tau_*^ka$ is contained in the image of $H_*(\tilde X_{\tilde\theta})\to H_*(\tilde X)$.
Thus $\tau_*^ka$ is contained in $\iota_*(D_+)$, see Lemma~\ref{L:7}.
Since $H_*(X_\theta)$ is assumed to be a finite dimensional vector space, we have
$\iota_*(D_-)=\iota_*(D)=\iota_*(D_+)$, see~\eqref{E:12}. Using Lemma~\ref{L:7} we thus conclude 
$\tau_*^ka$ is contained in the kernel on the right hand side of \eqref{E:543}.
Since this common kernel is invariant under the isomorphism
$\tau_*\colon H_*(\tilde X)\to H_*(\tilde X)$, we conclude that $a$ has to be contained in the common
kernel too, whence the theorem.
\end{proof}

Clearly, Theorem~\ref{T:monreg} and Proposition~\ref{P:C} imply Theorem~\ref{P1}.

\section {Example}\label{S7}

Figure~\ref{cv-ex} below, describes a tame angle valued map $f\colon X\to \mathbb R$ whose bar codes and Jordan cells are given in the attached table. 
\begin{figure}[h]
\includegraphics{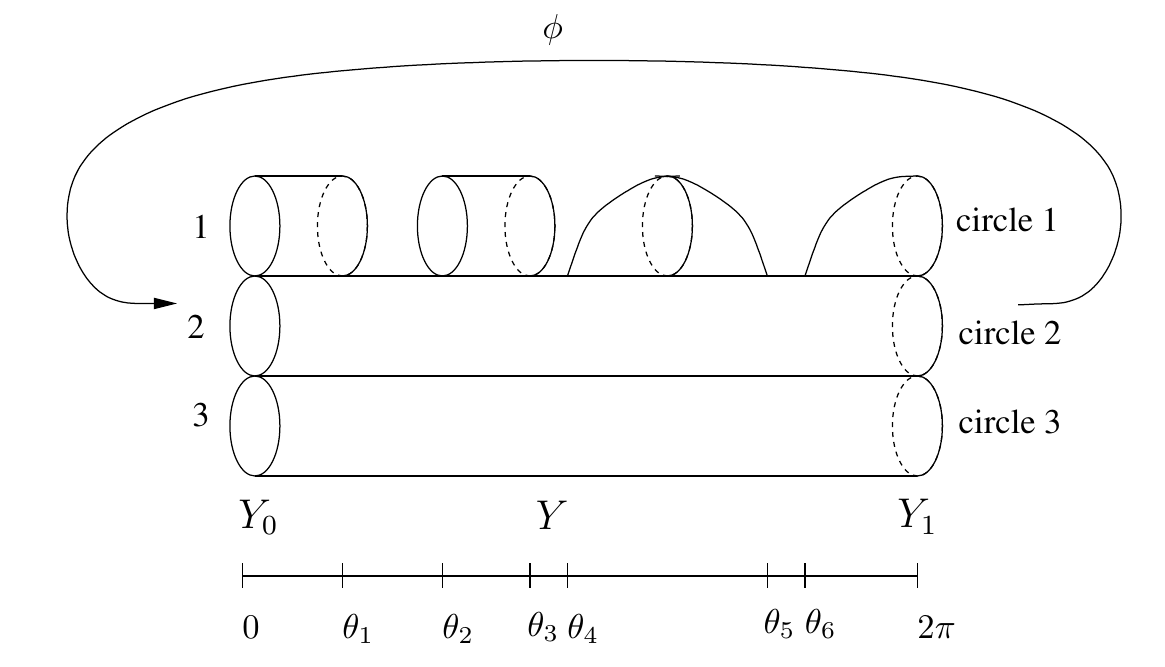}
~\\

{\scriptsize
\begin{tabular}{c|c}
map $\phi$ & $r$-invariants \\ \hline
\begin{tabular}{l}
circle 1: 3 times around circle 1\\
circle 2: 2 time around circle 2 \\
circle 3: 1 time around 2 counter clockwise and 2 times around circle 3\\
\end{tabular} & 
\begin{tabular}{c|c|c}
dimension & bar codes & Jordan cells \\ \hline
0  &  & $(1,1)$ \\ \hline
   & $(\theta_6, \theta_1+2\pi]$ & $(2,2)$ \\
1  & $[\theta_2,\theta_3]$ & \\
   & $(\theta_4,\theta_5)$ &  \\
\end{tabular} \\ \hline 
\end{tabular}}
\caption{Example of $r$-invariants for a circle valued map}
\label{cv-ex}
\end{figure}

The space $X$ is obtained from $Y$  by identifying its  
right end $Y_1$ (a union of three circles) to the left end 
$Y_0$ (a union of three circles) following 
the map $\phi\colon Y_1\to Y_0$ explained in the table. The map $f\colon X\to S^1$ is induced by 
the projection of $Y$ on the interval $[0,2\pi]$.
Note that  $H_1(Y_1)=H_1(Y_0)=\kappa\oplus\kappa\oplus\kappa$
and $\phi$ induces a linear map in $H_1$-homology represented by the matrix 
\begin{equation*}
\begin{pmatrix}
3&0 & 0\\
1&2&-1\\
0 & 0&2            
\end{pmatrix}.
\end{equation*}
There are no bar codes or Jordan cells in dimension 2 since each fiber of $f$ is one-dimensional 
and, as all fibers are connected in dimension zero we have only one Jordan cell $\rho^{II}(1; 1).$It remains to describe the bar codes and the Jordan cells in dimension 1. For this example it is not  hard to derive them 
by applying the main theorems:

In view of Theorem \ref{P1}  we see  that the monodromy identifies to the regular part of the linear relation defined by the  linear maps $\omega_1=\begin{pmatrix}3&0&0\\0&2&-1\\0&0&2\\0&0&0\end{pmatrix}:\kappa^3\to \kappa^4$ and $\omega_2= \begin{pmatrix}0&0&0\\0&1&0\\0&0&1\\0&0&0\end{pmatrix}:\kappa^3\to \kappa^4.$ This regular part can be calculated using the definition in subsection \ref {SS:reg} which can be calculated and is the Jordan cell $(2,2).$

Theorem \ref{T3} 1. a. implies that there exists an open bar code $(4,5)$  and one closed bar code $[2,3].$   This by looking at the homology of various $X_{[a,b]}$ with $0\leq a\leq b\leq 2\pi.$  The same argument implies that we have an other bar code of the form  $(\theta_6, \theta_1+2\pi k].$ Theorem \ref{T1} a.  implies that $k=1$
and these are all bar codes.

We explain below how to use the elementary transformations described in section~\ref{ET} to derive the bar codes and the Jordan cells in the table above.

Note that $m=7$ and we have three representations to consider: $\rho_0$, whose all vector spaces are isomorphic to $\kappa$ and linear maps  identity, 
the representation $\rho_2$ which is trivial and the representation $\rho_1$ which has to be described and decomposed.

\emph{The $G_{14}$-representation $\rho_1$:}
Choose $t_1,\dotsc,t_7$ so that we have $0<t_7-2\pi<\theta_1<t_1<\theta_2<\cdots<\theta_6<t_6<\theta_7=2\pi$.
One has:
$$
V_{2i-1}
=\begin{cases}
\kappa^2 & \text{for $i=2,4,6$} \\
\kappa^3 & \text{for $i= 1,5,7,9$}
\end{cases}
\qquad
V_{2i}
=\begin{cases}
\kappa^2 & \text{for $i=4,5,6$} \\
\kappa^3 & \text{for $i=1,2,3,7$}
\end{cases}
$$
$$
\beta_i
=\begin{cases}
\Id & \text{for $i=2,5,7$} \smallskip\\
\left(\begin{smallmatrix}0&0\\1&0\\0&1\end{smallmatrix}\right) & \text{for $i=1,3$} \medskip\\
\left(\begin{smallmatrix}0&1&0\\0&0&1\end{smallmatrix}\right)  & \text{for $i=4,6$}
\end{cases}
\qquad
\alpha_i
=\begin{cases}
\Id & \text{for $i=1,3,4,6$} \smallskip\\
\left(\begin{smallmatrix}0&0\\1&0\\0&1\end{smallmatrix}\right) & \text{for $i=2$} \medskip\\
\left(\begin{smallmatrix}0&1&0\\0&0&1\end{smallmatrix}\right) & \text{for $i=5$} \medskip\\
\left(\begin{smallmatrix}3&0&0\\0&2&-1\\0&0&2\end{smallmatrix}\right) & \text{for $i=7$}
\end{cases}
$$
\vskip .1in

Using the elementary transformation we modify the representation $\rho_1$ into $\rho(1)$ then into $\rho(2)$ and finally into $\rho(3)$ keeping track of the elimination of bar codes. 

Precisely: 

1. Apply $T_1(5)$  and get $\rho(1)= T_1(5)(\rho_1).$


2. Apply $T_4(3)\cdot T_3(2)$ and get $\rho(2)= T_4(3)\cdot T_3(2)(\rho(1))$.  

3. Apply $T_1(7)\cdot T_4(1)$ and get $\rho(3)= T_1(7)\cdot T_4(1)(\rho(2)).$

In view of the Appendix to section 3, which describe what each elementary transformation does, it is easy to see that: 
\begin{enumerate}
\item $\rho(3)$ has  all $\alpha_i$ but $\alpha_7$ and $\beta_i$  the identity with $\alpha _7= \begin{pmatrix}2&-1\\0&2\end{pmatrix};$ hence no bar codes,
\item $\rho(2)$ has one bar code $(6,8],$  
\item $\rho(1)$ has two bar codes $(6,8]$ and $[2,3],$ 
\item $\rho_1$ has the bar codes $(6,8],$  $[2,3]$ $(4,5).$ 
\end{enumerate}

\end{document}